\documentclass{amsart}
\usepackage{graphicx} 
\usepackage{hyperref}
\usepackage{amsmath, amsthm, amssymb, amscd}
\usepackage{mathtools}
\usepackage{mathabx}
\usepackage{enumerate}
\usepackage{hyperref}
\usepackage{verbatim}
\usepackage{subcaption}
\usepackage{cite}
\usepackage{color}
\usepackage{esint}
\usepackage{tikz-cd}
\usetikzlibrary{shapes.geometric}
\usepackage[shortlabels]{enumitem}
 
\usepackage{bm}

\makeatletter
\newcommand{\labitem}[2]{%
\def\@itemlabel{\textbf{#1}}
\item
\def\@currentlabel{#1}\label{#2}}
\makeatother

\title{On the Resistance Conjecture}
\author{Sylvester Eriksson-Bique}
\email{sylvester.d.eriksson-bique@jyu.fi}
\address{Department of Math. and Stat.
P.O. Box 35 \\
FI-40014 University of Jyväskylä}
\thanks{We thank Mathav Murugan and Riku Anttila for extensive discussions on the problem and for asking great questions, which lead to identifying the Whitney blending technique. Both of them had observed earlier the importance of writing the cutoff Sobolev as a Poincar\'e inequality and that the cutoff Sobolev inequality would follow from the log-Caccioppoli inequality for functions $f$ with $f=0$ outside $2B$. Anttila was the first to publish an approach to cutoff Sobolev based on these ideas and proved partial results that came quite close \cite{rikuanttila}.  After the first version of this draft, Murugan also shared us an unpublished draft which he had written and shared with colleagues in 2018, where he proved a version of cutoff Sobolev for $f$ vanishing outside $2B$. Murugan also had the early idea of decomposing a function $f$ into two parts. The missing piece that we add to the resolution here is the Whitney blending method, along with a few smaller but crucial technical observations. We thank Pekka Koskela for many lectures years ago teaching us a very similar method, and Juha Kinnunen for advocating for developing a theory applicable to all exponents $p$. We also thank Ryosuke Shimizu for a careful reading of the draft of the paper and for his comments that helped correct several inaccuracies. Peking University hosted the author in Fall 2025, and we thank it for the opportunity to give a sequence of lectures there, during which some of the ideas of the present paper matured. The author was partially supported by the Research Council of Finland via the project \emph{GeoQuantAM: Geometric and Quantitative Analysis on Metric spaces}, grant no. 354241.}
\subjclass[2020]{Primary: 31C25, 31E05, 30L99; Secondary: 49Q15, 26B05, 60J60, 60G30, 46E35, 49J52, 53C23, 31C15, 28A12}
\keywords{Resistance conjecture, Harnack inequality, Poincar\'e inequality, Cutoff Sobolev inequality, Capacity bounds, $p$-energy measure, energy measure, Whitney covering, extensions of Sobolev functions, martingale dimension, index of Dirichlet form.}
\date{\today}

\usepackage[shortlabels]{enumitem}
 
\newtheorem{theorem}[equation]{Theorem}
\newtheorem{lemma}[equation]{Lemma}

\newtheorem{proposition}[equation]{Proposition}

\numberwithin{equation}{section}

\theoremstyle{definition}
\newtheorem{definition}[equation]{Definition}

\theoremstyle{remark}
\newtheorem{remark}[equation]{Remark}

\newtheorem{question}[equation]{Question}

    \newcommand*{\N}{\mathbb{N}}
    \newcommand*{\Z}{\mathbb{Z}}
    \newcommand*{\R}{\mathbb{R}}
        \DeclareMathOperator{\diam}{diam}

        \DeclareMathOperator{\rad}{rad}

        \DeclareMathOperator{\loc}{loc}

        \DeclareMathOperator{\Lip}{Lip}
        \DeclareMathOperator{\PI}{PI}
        \DeclareMathOperator{\cCap}{\rm Cap}
        
        \DeclareMathOperator{\cE}{\mathcal{E}}

        \DeclareMathOperator{\cF}{\mathcal{F}}
        
        \DeclareMathOperator{\cH}{\mathcal{H}}
        
        \DeclareMathOperator{\cB}{\mathcal{B}}

        \DeclareMathOperator{\LIP}{LIP}

        \DeclareMathOperator{\spt}{spt}

        \DeclarePairedDelimiter\norm{\lVert}{\rVert}
\begin{document}

\maketitle

\begin{abstract}
We give an affirmative answer to the resistance conjecture on characterization of parabolic Harnack inequalities in terms of volume doubling, upper capacity bounds and a Poincar\'e inequalities.  The key step is to show that these three assumptions imply the so called cutoff Sobolev inequality, an important inequality in the study of anomalous diffusions, Dirichlet forms and re-scaled energies in fractals. This implication is shown in the general setting of $p$-Dirichlet Spaces introduced by the author and Murugan, and thus a unified treatment becomes possible to proving Harnack inequalities and stability phenomena in both analysis on metric spaces and fractals and for graphs and manifolds for all exponents $p\in (1,\infty)$. As an application, we also show that a Dirichlet space satisfying volume doubling, Poincar\'e and upper capacity bounds has finite martingale dimension and admits a type of differential structure similar to the work of Cheeger. In the course of the proof, we establish methods of extension and characterizations of Sobolev functions by Poincar\'e-inequalities, and extend the methods of Jones and Koskela to the general setting of $p$-Dirichlet spaces.
\end{abstract}

%{\color{blue} RA Comment to abstract: In my opinion, Kumagai should not be mentioned in resistance conjecture. He posed an open problem to simplify heat kernel estimates. Barlow is also a bit questionable I think. But the question they basically posed was definitely resistance conjecture so mentioning them is not wrong, but a bit questionable to relate them to resistance conjecture.\\
%MM: I agree with Riku. Perhaps the precise statement of the should be attributed Grigor'yan, Hu and Lau. Barlow and Kumagai were vague and it is not clear whether or not they meant the resistance conjecture.}

\section{Introduction}
\subsection{Main results}
The resistance conjecture was posed by Grigor'yan, Hu and Lau in \cite[Conjecture 4.16]{GHL},  \cite[p. 1494]{GHL2}, \cite[Conjecture 4.28]{rikuanttila}. It asks if volume doubling, resistance bounds and a Poincar\'e inequality suffice to prove a parabolic Harnack inequality for general metric measure Dirichlet spaces.  Similar, and less specific, general questions seeking to find minimal conditions that characterize parabolic Harnack inequalities were posed in \cite[Open problem II]{kumagaiICM} and \cite[Remark 3.17]{AnalysisBarlow}. In \cite{BarlowBassStability, barlow2006stability}, it was shown that Poincar\'e inequalities, doubling and the cutoff Sobolev inequality were equivalent to a parabolic Harnack inequality. The cutoff Sobolev inequality was quite mysterious and numerous authors commented on the difficulty of verifying it for explicit spaces. Further, it was unclear what its geometric content was. Here, we remove the cutoff Sobolev inequality from the characterization by showing that it is implied by volume doubling, Poincar\'e and upper capacity bounds thus leading to a conceptually simple generalizations of the seminal result of Grigor'yan and Saloff-Coste in \cite{grigor,saloff-coste}, simplifying the characterization in \cite{barlow2006stability}.
\begin{theorem}\label{thm:mainthmhke}
 Let $(X,d,\mu,\cE,
 \cF)$ be a metric measure Dirichlet space, then the following are equivalent:
 \begin{enumerate}
     \item $(X,d,\mu,\cE,
 \cF)$ satisfies the parabolic Harnack inequality ${\rm PHI}(\beta)$.
     \item $(X,d,\mu,\cE,
 \cF)$ satisfies  heat kernel estimates ${\rm HKE}(\beta)$ and volume doubling of $\mu$.
     \item $(X,d,\mu,\cE,
 \cF)$ satisfies volume doubling of $\mu$, elliptic Harnack ineqality and a two-sided capacity bound.
     \item $(X,d,\mu,\cE,
 \cF)$ satisfies volume doubling of $\mu$, Poincar\'e inequality $\PI(\beta)$ and capacity upper bound $\cCap_{2}(\beta)$.
 \end{enumerate}
\end{theorem}
Notice that here, and below, we avoid any assumptions on the metric being geodesic. For definitions of parabolic Harnack inequalities, two-sided capacity bounds (\cite[Definition 4.2]{MN}), Heat kernel estimates and metric measure Dirichlet spaces, see e.g. \cite{GHL, MN, erikssonmurugan} - we will not use these directly in our proofs and thus do not discuss them in more detail here; see \cite{Grigoryan,kumagaiICM,GHL,strichartz} for clear surveys and \cite{MN,rikuanttila} for well-written overviews on the literature and background.  This result is a combination of a fairly large array of papers, and we recommend to see \cite[Theorem 4.5]{MN} for a more detailed overview of the proof of the known parts. We briefly recall the history. In $\R^n$, elliptic and parabolic Harnack inequalities were established by Moser \cite{mosera,moserb}. This was extended to manifolds with non-negative Ricci curvature in \cite{CY, LY, Yau}. If $\beta=2$, then $\cCap_{2}(\beta)$ is automatically satisfied and the result is known from earlier works, in the manifold case \cite{grigor,saloff-coste,grigoryansaloffcoste}. The present version, with the addition of the cutoff Sobolev inequality was largely given in \cite{barlow2006stability}, while the case for graphs was studied in \cite{BarlowBassStability,GT}. 

Proving that the cutoff Sobolev inequality is implied by volume doubling and Poincar\'e has also other applications. Indeed, the cutoff Sobolev inequality has been used as a condition for a variety of results: e.g. on singularity phenomena \cite{yangenergy,MN},  reflected processes\cite{muruganreflected,anttilareflected} and giving bounds for martingale dimension \cite{muruganmartingale}. We note one such application here.
\begin{theorem}\label{thm:diffstruct}
    If a metric measure Dirichlet space satisfies volume doubling, $\cCap_p(\beta)$ and a Poincar\'e inequality $\PI_p(\beta)$, then its martingale dimension is finite and it admits a measurable Riemannian structure in the sense of \cite[Theorems 3.1 and 3.4]{hinodiff}.
\end{theorem}
The proof of this is obtained directly by applying Theorem \ref{thm:mainthmhke} to obtain heat kernel estimates and then applying \cite[Theorem 5.7]{erikssonmurugan} and \cite[Theorems 3.1 and 3.4]{hinodiff}. Martingale dimension (or index of a Dirichlet form) goes back to \cite{DV,MT} and for earlier work see  \cite{kusuoka, HinoM,HinoEnergy,Hinoupper}. It is a difficult concept to understand, very difficult to bound and has several suprising properties. From this perspective the previous theorem gives a particularly appealing simple condition for finiteness of martingale dimension and the existence of a type of differential structure that is very analogous to the seminal result of Cheeger \cite{Ch}. Cheeger considered analytic dimension and Lipschitz differentiability (a type of Rademacher theorem for Lipschitz functions), while here we consider martingale dimension and a type of differentiability coming from \cite[Theorem 3.4]{hinodiff}. Theorem \ref{thm:diffstruct} thus can be seen as a wide ranging generalization of \cite{Ch}. See \cite{erikssonmurugan,muruganmartingale} for further discussion and the definition of martingale dimension, and \cite{hinodiff} for a discussion on the associated differential structure. See also \cite{teriseb} for a closely related discussion on differential structures associated to upper gradients.
For further motivation on the cutoff Sobolev inequality see \cite{murugannonlocal}. 

Both of the above theorems follow from establishing the cutoff Sobolev inequality from the other conditions. This part can formulated for all exponents $p\in (1,\infty)$ and our methods become more transparent in this setting.  Thus, instead of the Dirichlet space setting, we will next focus on $p$-Dirichlet spaces introduced in \cite{erikssonmurugan}. This allows for a unified treatment of the cutoff Sobolev estimate in a large variety of energy constructions. An additional motivation comes from applications to other problems in the literature. Similar fractal $p$-energies have recently been extensively studied  \cite{murugan2023first,kigami,shimizu, ASEBShimizu,yangenergy,KS,strichartzp,pcfpenergy}, and by working in this setting, we establish the first step to extending Theorems \ref{thm:mainthmhke} and \ref{thm:diffstruct} for all exponents $p>1$. We note that constructions of discrete modulus, discrete capacities, $p$-variation measures for functions, Sobolev spaces and Besov spaces have also played a role in geometric group theory \cite{P89,BS,BourdonPajot,pallierpvar}, but that their connection to these self-similar $p$-energies remain unstudied. 

A $p$-Dirichlet space is a tuple $(X,d,\mu,\cE_p,\cF_p, \Gamma_p)$, where $(X,d,\mu)$ is a complete locally compact metric measure space equipped with a Radon measure $\mu$. A function space $\cF_p\subset L^p(X)$ is defined as the domain of a (closed) energy $\cE_p$, which is local and given by integration of an energy measure $\Gamma_p\langle f\rangle$ defined for each $f\in \cF_p$. Detailed axioms are in Section \ref{sec:pdirspace}. As shown in \cite{erikssonmurugan}, any metric measure Dirichlet space on a locally compact metric measure space defines a $2$-Dirichlet structure. Further, as observed there, such structures also arise naturally via upper gradients in analysis on Metric spaces and via rescaled energies, see e.g. \cite{erikssonmurugan, murugan2023first,shimizu, kigami,ASEBShimizu}.

The two key assumptions are the $(p,\beta)$-upper capacity bound $\cCap_p(\beta)$ and the $(p,\beta)$-Poincar\'e inequality $\PI_p(\beta)$. In these definitions, we set $f_B=\frac{1}{\mu(B)}\int_B f d\mu$.\footnote{We also use $\psi_B$ to denote a functions indexed by a ball $B$, but this convention is only used with Greek letters in this paper, for which we will not have a need to take an average.}

\begin{enumerate}
    \item  $\cCap_{p}(\beta)$: There exists a constant $C_{\rm cap}>0$ s.t. the following holds. For every $B(x,r)\subset X$ there exists a $\psi_B \in C (X)\cap \cF_p$ with $\psi_B|_B=1$ and $\psi_B|_{X\setminus 2B}=0$ and
    \[
    \Gamma_p\langle \psi_B \rangle(X)=\Gamma_p\langle \psi_B \rangle(2B)\leq C_{\rm cap}\frac{\mu(B(x,r))}{r^\beta } 
    \]
    \item $\PI_p(\beta)$: There exist constants $C_{\rm PI}, \Lambda>0$ s.t. the following holds. For every $B(x,r)\subset X$ and every $f\in \cF_p$
    \[
    \fint_B |f - f_B|^p d\mu \leq C_{\rm PI}\frac{r^\beta}{\mu(B(x,r))}\Gamma_p\langle f\rangle(\Lambda B). 
    \]
\end{enumerate}

%or simplicity assume also that $X$ is geodesic. All should work for locally linearly connected spaces: For all $x,y\in X$ there exists a connected set $E$ s.t. $\diam(E)\leq Cd(x,y)$. In fact, the Poincar\'e inequality must imply this with roughly the usual argument from Cheeger \cite{Ch} with chains of points.

As stated, the main task in proving Theorems \ref{thm:mainthmhke} and \ref{thm:diffstruct} is understanding the \emph{generalized capacity condition/cutoff Sobolev inequality}. The cutoff Sobolev inequality has so far appeared in several different forms in the literature, see below for some discussion on this. In \cite{barlowanders}, see also \cite{MN,yangenergy,rikuanttila}, it settled to the following form: There exists a constants $C,\Lambda \geq 1$ s.t. for every $B(x,r)$ there exists a $\psi_B \in \cF_2$ with $\psi_B|_B=1$ and $\psi_B|_{X\setminus 2B}=0$ s.t. for every $f\in \cF_p$ and every quasicontinuous representative $\tilde{f}$ of $f$ we have
\[
\int_{2B} \tilde{f}^2 d\Gamma_2\langle \psi_B\rangle \leq C \Gamma_2\langle f\rangle(\Lambda B) + C\frac{1}{r^\beta}\int_{2B} f^2 d\mu. 
\]
 In this form, the inequality does not appear to have a familiar analogue in classical PDE theory or in nonlinear potential theory. Indeed, in Euclidean spaces this inequality is a triviality with $\psi_B(x)=
\min(1,\frac{d(x,X\setminus 2B)}{r})$. Note that a $p\neq 2$ analogue was considered by Yang in \cite{yangenergy}. 

A starting point for our work is the fact noticed by a few authors (e.g. Murugan and Anttila) in recent years that one could write the cutoff Sobolev inequality in an different equivalent form, which lends itself to an easier analysis.  For the equivalence of these two forms, see Lemma \ref{lem:equivalence} below. We say that the structure satisfies the $p$-cutoff Sobolev inequality ${\rm CS}_{p}$, if the following holds. There exists constants $C_{\rm CS},\Lambda \geq 1$ s.t. for every $B(x,r)$ there exists a $\psi_B \in \cF_p$ with $\psi_B|_B=1$ and $\psi_B|_{X\setminus 2B}=0$ s.t. for every $f\in \cF_p$ and every quasicontinuous representative $\tilde{f}$ of $f$ we have
\begin{equation}\label{eq:CS}
\int_{2B} |\tilde{f} - f_B|^p d\Gamma_p\langle \psi_B\rangle \leq C_{\rm CS}\Gamma_p\langle f\rangle(\Lambda B). 
\end{equation}
Written in this way, the inequality becomes independent of $\beta$ and gains a much more familiar form: it is a type of two-weighted Poincar\'e inequality. The important difficulty here, however, is that $\Gamma_p\langle \psi_B\rangle$ can be singular with respect to the reference measure $\mu$, see e.g. \cite{kusuoka,KCsing,NMsing,ASEBShimizu, KS, yangenergy,murugan2023first} for discussion on the singularity phenomena. In fact, this difficulty with singularity is evident in the inequality: the average $f_B$ on the left hand side is computed with respect to the measure $\mu$ and not the possibly singular measure $\Gamma_p\langle \psi_B\rangle$. This remark is quite important, since $\Gamma_p$ may fail to be doubling (see e.g. \cite[Theorem 8.4]{ASEBShimizu} for examples where it is not), and it is an important technical point that reoccurs in our proofs. Indeed, in the extension methods below, we also use $\mu$-averages throughout. A reader interested in this technical point may find the related two-weighted discussion of \cite{bjornkalamajska}, where singularity is also possible, quite illuminating. See also therein and \cite{rikuanttila} for a longer discussion on the connection to other inequalities and trace theorems. 

The weighted Poincar\'e inequality form suggests a connection to techniques from the field of analysis on metric spaces, an idea that was first stated by Riku Anttila in \cite{rikuanttila}. Recognizing this connection leads to a wealth of techniques one could try from analysis on metric spaces:  see  \cite{kinnunenkorte,mazya,sobmet,HajKo} for a paper whose techniques inspire us here, \cite{HKST,BB,He} for general introductions, \cite{HK,Ch,Shanmun} for some influential early papers and \cite{HeNonSmooth,MRC, KleinerICM,BonkICM} for some relevant surveys on analysis on metric spaces.  Using this form of the cutoff Sobolev inequality, we prove the following.
%{\color{red} Where introduced? At least \cite{barlowanders}, see also \cite{GHL2}.}
%{\color{blue} RA: CS was first in \cite{BarlowBassStability} on graphs and then \cite{barlow2006stability} on MMD spaces. These had $\delta > 0$ like I did. $\delta = 0$ was first in \cite{barlowanders}. But I believe it was never written as a Poincar\'e inequality before my paper and it was interpreted as a ``weighted Sobolev inequality''.
%In my opinion, this Poincar\'e inequality interpretation is one key part of the resolution and it is a one clear connection to Analysis on metric spaces.}

\begin{theorem}\label{thm:mainthm} If $(X,d,\mu,\cF_p,\Gamma_p)$ is a regular local $p$-Dirichlet structure, with $\mu$ doubling, and satisfying $\PI_p(\beta)$ and $\cCap_p(\beta)$. Then, $(X,d,\mu, \cF_p,\Gamma_p)$ satisfies the $p$-cutoff Sobolev inequality ${\rm CS}_p$.  
\end{theorem}
In fact, the proof shows that one can use harmonic $\psi_B$ as cutoff functions, i.e. capacity minimizers. This fact is new even for manifolds and graphs for $\beta=p$ and in analysis on metric space. Previously, cutoff functions were constructed as truncated and scaled distance functions or by using resolvents. A difficulty in using harmonic functions in the cutoff Sobolev inequality is that they may fail to be Lipschitz even in fairly nice settings such as  doubling spaces with lower curature lower bounds (e.g. a cone $C(\alpha S)$ \cite{bbi} over the metric space $\alpha S$, where $\alpha S$ is a circle of length $2\pi \alpha$ and $\alpha>1$). In fractals, Lipschitz functions may be completely singular with respect to Sobolev spaces \cite[Theorem 1.1]{ASEBShimizu}, and thus distance functions are useless in the general case. As a fairly direct corollary of Theorem \ref{thm:mainthm}, by combining with results \cite[Theorem 4.5]{MN}, we get a proof of Theorem \ref{thm:mainthmhke}.

In previous literature, the cutoff Sobolev inequality was used to adapt the Moser iteration technique to fractals and prove elliptic and parabolic Harnack inequalities. This played an important role in establishing the stability of these properties in \cite{barlowmurugan, barlow2006stability, BarlowBassStability,barlow2020stability,barlowanders}. We established now the cutoff Sobolev inequality for all $p>1$ and this suggests that many of these results have an extension for all $p>1$.  Theorem \ref{thm:mainthm} is the first step in showing this in the general case, while some specific cases have been studied in \cite{murugan2023first,yangenergy}. The simpler characterization of the cutoff Sobolev inequality has also other applications. For example, one can use it to deduce singularity phenomena \cite{yangenergy,MN}. It also plays a role in the study of reflected processes \cite{muruganreflected}, many blow-up arguments and estimates for martingale dimension, see Theorem \ref{thm:diffstruct}.

The main technique and new idea in the proof of Theorem \ref{thm:mainthm} is a blending technique: if $f,g\in \cF_p$ and $B$ is a ball, then there exists a function $h\in \cF_p$ s.t. $h|_B=f|_B$ and $h|_{X\setminus 2B}=f|_{X\setminus 2B}$. Such a blending is known to be possible assuming the cutoff Sobolev inequality, but the novelty here is that one can prove it directly from capacity bounds and a Poincar\'e inequality by adapting an old method of Whitney \cite{whitney} via extension domains in \cite{jones}, \cite{tuominen}. Such Whitney covers have been frequently used (in particular \cite{muruganreflected}), but not in the study of the cutoff Sobolev inequality. This method is general, and appears useful in other problems as well. Indeed, Mathav Murugan has already adapted this method to resolve a non-local version of the resistance conjecture in \cite{murugannonlocal}.

We have written this introduction with slightly simpler statements. The main theorems have, however, a natural quasisymmetric invariance, and motivated by \cite{barlow2020stability, barlowmurugan} we aim to write all of our proofs in a quasisymmetrically invariant form. This leads to a slightly stronger formulation of the main theorem in Theorem \ref{thm:mainthm-sharper}. The proofs are not any more complicated by this generality, and only some care is needed in the notation. 

We end this overview with an open problem. The Poincar\'e inequality is in some cases almost as difficult to prove as the cutoff Sobolev inequality. However, in many settings one can show quite easily the following balled-capacity estimate. 
\vspace{3pt}
\noindent ${\rm Cap_{p,Ball}(\beta):}$ For all $A\geq 3$ there exists a $C_A>1$ s.t. for all $x,y\in X, x\neq y$ and $r=d(x,y)$ and all $\phi\in \cF_p$ s.t. $\phi|_{B(x,r/A)}=1, \phi|_{B(y,r/A)}=0$ we have
\[
\cE_p(\phi)\geq C_A \frac{\mu(B(x,r))}{r^\beta}.
\]
\begin{question}
Do $\cCap_p(\beta), {\rm Cap_{p,Ball}}(\beta)$ and volume doubling suffice to prove a Poincar\'e inequality? 
\end{question}
In fact, for $p\neq 2$ it is unkown if this holds even for quite simple fractals, such as the generalized Sierpi\'nski carpets. 
We consider Theorem \ref{thm:mainthm} as partial evidence towards a positive answer. Entirely new and different methods seem to be needed for its proof, however. 
If true, this would lead to an even simpler condition for elliptic and parabolic Harnack inequalities, and answer another problem \cite[Problem 7.7]{barlowmurugan}. Currently, a method of Bonk and Kleiner \cite[Proposition 3.1]{BK04} and Heinonen and Koskela \cite[Corollary 5.13]{HK} as adapted by Murugan and Shimizu \cite[Sections 3 and 4]{murugan2023first} indicates that this is true under the additional constraint of $\beta>Q-1$, when $\mu(B(x,r))\sim r^Q$. The general case is not known. It would help in proving Harnack inequalities for a variety of symmetric constructions for which it is presently not known, see e.g. \cite{anttilaseb, BourK,clais, debseboned} for some examples of such spaces. We next discuss the literature on the cutoff Sobolev inequality in more detail and give an outline the main proof ideas.

\subsection{Background on the cutoff Sobolev inequality}
The cutoff Sobolev inequality has been studied in quite a large number of works \cite{yangenergy, BarlowBassStability, barlow2006stability, barlowanders, rikuanttila}, and has appeared in many different forms. Here, we briefly aim to cite the main versions and discuss their relationships. The condition was introduced in \cite[Definition 1.4]{BarlowBassStability} for graphs in the form ${\rm CS}^\delta$: for all balls $B=B(x,r)$ there exists a H\"older continuous $\psi_B\in \cF_2$ s.t $B(y,s)\subset B(x,2r)$
\[
\int_{B(y,s)} \tilde{f}^2 d\Gamma_2\langle \psi_B\rangle \leq C \left(\frac{s}{r}\right)^\delta \left( \Gamma_2\langle f\rangle(\Lambda B) + C\frac{1}{s^\beta}\int_{2B} f^2 d\mu\right). 
\]
One could call this the strong cutoff Sobolev inequality, since it involves an energy decay for smaller balls. It also assumes the existence of a H\"older continuous cutoff functions, which is undesirable, as one would not like to assume any a priori regularity for Sobolev functions.

In \cite{barlowanders,barlow2006stability}, the condition on all scales was dropped and the inequality obtained the present form. As observed in \cite{barlowanders,yangenergy} the cutoff Sobolev inequality enjoys a self-improvement property, where one can show for every $\theta \in (0,1)$ there is a $C_\theta>0$ s.t. one can construct a cutoff function $\psi_B$ s.t. 
\[
\int_{2B} \tilde{f}^2 d\Gamma_2\langle \psi_B\rangle \leq \theta \Gamma_2\langle f\rangle(\Lambda B) + C_\theta \frac{1}{r^\beta}\int_{2B} f^2 d\mu. 
\]
That is, one can adjust the constants on the right hand side. 

In \cite{rikuanttila}, a simpler condition emerged. One says that the cutoff energy inequality  is satisfied if there exists a $\delta>0$ s.t. for all balls $B=B(x,r)$ there exists a cutoff function $\psi_B$ as before for which $\Gamma_2\langle \psi_B\rangle(B(y,s)) \lesssim \left(\frac{s}{r}\right)^\delta\frac{\mu(B(y,s))}{r^\beta}$. A desirable aspect of this estimate is that it only involves the cutoff function itself and not any condition on all functions $f\in \cF_2$. This implies ${\rm CS}^\delta$, and is further implied by EHI and Poincar\'e inequalities and capacity bounds. See \cite{rikuanttila} for more details. 

For a space satisfying volume doubling and Poincar\'e inequalities and for $p=2$, each of these inequalities are known to be equivalent, since each of them can be used to prove a parabolic Harnack inequality, which in turns implies any of them.  We would expect that all of them are also equivalent for $p>1$. This would from establishing elliptic Harnack inequalities, under the assumption of volume doubling, upper capacity bounds and Poincar\'e, and the results from \cite{rikuanttila}.

\subsection{Key proof ideas}
Theorem \ref{thm:mainthmhke} will follow from Theorem \ref{thm:mainthm} and \cite[Theorem 4.5]{MN} via very short argument. From these Theorem \ref{thm:diffstruct} is a direct corollary, as explained above. Thus, we focus on Theorem \ref{thm:mainthm}. The proof of Theorem \ref{thm:mainthm} rests on two observation. The first observation, already noted by some specialists (e.g. Murugan and Anttila) earlier, is that the cutoff Sobolev inequality is easier to show if $f=0$ outside $2B$. This follows from an argument that we learned from \cite[Theorem 3.5]{kinnunenkorte} via a simple energy inequality in Lemma \ref{lem:LogCaccioppoli}. This lemma is often proven as a consequence of the so called logarithmic Caccioppoli inequality, but it follows here from a simple argument using the variational principle and locality. This argument actually goes back to at least Maz'ya \cite[e.g. Corollary on p. 113]{mazya} and involves the so called Maz'ya truncation technique \cite[p. 110]{mazya} and \cite[Lemma 3.3]{kinnunenkorte}. For similar and related discussion see \cite{bakrysobolev,sobmet,saloffcostebook}, and the very illuminating article of Maz'ya \cite{mazyalecture}.

The second observation is a way to reduce effectively to the case similar to $f=0$ outside $2B$. Indeed, in an attempt to do so, one is led to show the following: there exists a function $h\in \cF_{p}$ s.t. $h|_{3/2 B}=f|_{3/2B}$ and $h|_{X\setminus 2B}=f_B$. (Technically, when $f_B\neq 0$, and the space is unbounded, we would need to consider a local Sobolev space here.) Such a function $h$ would at least capture a part of the function $f$ and have the desired boundary data. In this case, we could compute:
\[
\int_{3/2 B} |f-f_B|^p d\Gamma_p\langle \psi_B\rangle = \int_{3/2 B} |h-f_B|^p d\Gamma_p\langle \psi_B\rangle \lesssim \Gamma_p\langle h\rangle(\Lambda B).
\]
The last estimate would follow from the first observation. If we could now control the energy of $h$ in terms of the energy of $f$, we would at least have the desired with $3/2B$ replacing $2B$. We will return to the part $2B\setminus (3/2B)$ shortly, but now focus on the construction of $h$.

The problem of constructing such a function $h$ arises naturally in many settings, where it would be convenient to localize a Sobolev function. Classically, one would construct a function $h$ by first choosing a function $\psi_B\in \cF_p$ with  $\psi_B|_B=1$ and $\psi_B|_{X\setminus 2B}=0$ and setting $h=\psi_B f + (1-\psi_B) f_B$. The problem here: the control of the energy of $h$ involves a term of the form $\int_{2B} |\tilde{f}-f_B|^p d\Gamma_p\langle \phi\rangle$, and it brings us right back to the cutoff Sobolev inequality and an argument using this approach seems circular. (See also the beginning of Section \ref{sec:whitneyblending} for an argument along this line assuming the cutoff Sobolev inequality true.)

Curiously, there is a different way to construct $h$ that does not rely on the cutoff Sobolev inequality. The general form of the statement is given in the following lemma.

\begin{lemma}\label{lem:whitneyblending} Assume that $(X,d,\mu,\cF_p,\Gamma_p)$ is a regular local $p$-Dirichlet structure, with $\mu$ doubling, and satisfying $\PI_p(\beta)$ and $\cCap_p(\beta)$. For every $\eta\in (0,1)$ there exists a $C_{\rm WB}=C_{\rm WB}(C_{\rm cap}, C_{D}, C_{\PI})>0$ s.t. the following holds. Let $f,g\in \cF_p$ and $B_0$ is a ball in $X$, then, there exists a $h\in \cF_p$ such that $h|_{B_0}=f|_{B_0}$, $h|_{X\setminus (1+\eta)B_0} = g|_{X\setminus (1+\eta)B_0}$ with
\[
\Gamma_p\langle h\rangle(2B_0)\leq C_{\rm WB}\frac{\mu(B_0)}{\rad(B_0)^\beta} \fint_{(1+\eta)B_0}|f-g|^pd\mu + C_{\rm WB} \Gamma_p\langle f\rangle(2B_0)+C_{\rm WB}\Gamma_p\langle g\rangle(2B_0)
\]
\end{lemma}
The proof of this lemma involves adapting techniques from Sobolev-extension problems \cite{jones,tuominen, whitney,muruganreflected} and using averages with respect to a Whitney cover. We call this method \emph{Whitney blending}. The rough idea of using averages here is that it leads to energy terms that are more directly controlled by capacity bounds and the Poincar\'e inequality. Showing that the resulting function is in $\cF_p$ needs a new argument that adapts a characterization of Sobolev functions in terms of Poincar\'e inequalities and is inspired by \cite{sobmet}. There are quite a few other details here and the proof  will be the most technical part of the paper, although the definition of $h$ itself is explicit and somewhat simple. See Section \ref{sec:whitneyblending}, for more explanation of the ideas and references. 

Lemma \ref{lem:whitneyblending} is the most important tool in our proof the cutoff Sobolev inequality. After this lemma, only one final idea is needed: to decompose $f$ into two parts, $h_1$, which equals $f$ on $3/2B$ and equals $0$ outside $X\setminus 2B$, and $h_2$, which equals $f$ outside $3/2 B$ and equals $0$ on $B$. One can bound $|f|\leq |h_1|+|h_2|$ on $2B$ directly, and apply a version of the argument from the first paragraph in this subsection to each $h_1,h_2$. It is helpful to note that the estimates for $h_1$ and $h_2$ can be seen as somewhat symmetric in that $B$ and $X\setminus 2B$ switch roles. To formalize this idea of symmetry, however, forces us to consider also local Sobolev spaces in the proofs below.

We end this introduction with highlighting the main elements of the philosophy of our proofs: adopting a perspective of all $p\neq 2$ and comprehensively applying techniques from analysis on metric spaces in the study of Dirichlet forms. There have been several other recent advances with a very similar strategy, see e.g. \cite{erikssonmurugan, barlow2020stability, anttila2024constructions, rikuanttila}. Here, the adapted techniques are that of Maz'ya truncation, capacity integration, characterization of Sobolev functions in terms of Poincar\'e inequalities and Sobolev extension. 

\subsection{Outline} In Section \ref{sec:pdirspace}, we give the definitions and basic properties of $p$-Dirichlet spaces. Note that, by \cite{erikssonmurugan} a Dirichlet space satisfies these assumptions for $p=2$, and a reader familiar with this theory may skip this section. In Section \ref{sec:proofofmain} we show how Lemma \ref{lem:whitneyblending} leads to a proof of the main theorems. The most technical part of the paper is Section \ref{sec:whitneyblending}, where we prove Lemma \ref{lem:whitneyblending}. This proceeds through a sequence of auxiliary statements concerning Whitney covers, partitions of unity, and characterizations of functions $f\in \cF_p$.

\subsection*{Basic Notation}
Throughout $(X,d,\mu)$ will be a complete, locally compact, separable metric space, and  $\mu$ will be a Radon measure on $X$. If $\nu$ is another Radon measure, then we write $\mu \leq \nu$ if $\mu(A)\leq \nu(A)$ for all Borel sets $A\subset X$.

Open balls with center $x$ and radius $r>0$ are denoted $B(x,r)=\{y\in X : d(x,y)<r\}$ and the inflation of a ball $B=B(x,r)$ is denoted by $CB=B(x,Cr)$, for $C>0$. The center of a ball $B$ is denoted $x_B$ and its radius $\rad(B)$. As sets, two different balls may coincide, but we will always assume that a ball comes with a specified center and radius. For $A\subset X$, define $\diam(A)=\sup_{x,y\in A} d(x,y)$. If $\diam(X)<\infty$, we will only consider balls $B$ with radii $\rad(B)\leq 2\diam(X)$. If $A,B\subset X$ are not empty, define $d(A,B)=\inf_{a\in A, b\in B} d(a,b)$. If $x\in X$, we also write $d(x,A)=\inf_{a\in A} d(x,a)$.

The measure $\mu$ is usually assumed $C_D$-doubling: $\mu(B(x,2r))\leq C_D\mu(B(x,r))$, and this implies that $X$ is proper (all closed and bounded sets are compact) and $C_M=C_M(C_D)$-metrically doubling (for every ball $B\subset X$ there exists $x_1,\dots, x_{C_M} \in B$ s.t. $B\subset \bigcup_{i=1}^{C_M} B(x_i, \rad(B)/2)$). A space $(X,d,\mu)$ is called volume doubling if $\mu$ is a doubling measure. For a measurable subset $A\subset X$ of finite and positive measure define $f_A = \frac{1}{\mu(A)}\int_A f d\mu$, whenever the integral is defined.

Continuous functions are denoted $C(X)$, and compactly supported continuous functions are denoted $C_c(X)$. These spaces are equipped with the supremum norm $\|f\|_\infty = \sup_{x\in X}|f(x)|$. We denote the $L^p$-spaces by $L^p(X)=\{f:X\to [-\infty,\infty] : f \text{ measurable and } \int_X |f|^p d\mu <\infty\}$ with their norm $\|f\|_{L^p(X)}=\left(\int_X |f|^p d\mu\right)^\frac{1}{p}$. $\Lip(\R^n)$ is the collection of Lipschitz functions (there exists a constant $L\geq 0$ s.t. $|f(x)-f(y)|\leq Ld(x,y)$ for all $x,y\in X$) and $\Lip_0(\R^n)$ consists of those $f\in \Lip(\R^n)$ with $f(0)=0$. Write $\LIP[f]=\sup_{x,y\in X, x\neq y} \frac{|f(x)-f(y)|}{d(x,y)}$ and $\Lip_a[f](x)=\lim_{r\to 0} \LIP[f|_{B(x,r)}]$. The function $x\to \LIP[f|_{B(x,r)}]$ is upper semi-continuous for every $r>0$, and thus $x\to \Lip_a[f](x)$ is Borel measurable.

We will use $C_{\bullet}$ and $\Lambda$ to denote all \emph{structural constants}, e.g. the constants involved in doubling, capacity bounds and Poincar\'e inequalities. We write $f\lesssim g$, if $f\leq Cg$, where $C$ depends only on the structural constants. When introducing a new structural constant, which depends on old ones, we write $C_{\bullet}=C_\bullet(C_1,C_2,\dots)$ to indicate the previous structural constants in terms of which it can  be bounded by a coordinate-wise monotone relationship.

\section{$p$-Dirichlet spaces}\label{sec:pdirspace}

\subsection{Axioms and basic properties}
The axioms for a regular $p$-Dirichlet space were given in \cite{erikssonmurugan}, see therein for more discussion on the definition.

\begin{definition} \label{d:dir-space}
		We call $(X,d,\mu, \mathcal{E}_p, \mathcal{F}_p, \Gamma_p)$  a regular local $p$-Dirichlet space when the following properties hold:
		\begin{enumerate}  
			\item \textbf{Locally compact}: $(X,d,\mu)$ is a locally compact metric space equipped with a Radon measure $\mu$.
			\item \textbf{Completeness:} 	The space $\mathcal{F}_p$  is a subspace of $L^p(X,\mu)$ and
			$\mathcal{E}_p: \mathcal{F}_p \to [0,\infty)$ is a non-negative function such that $\mathcal{F}_p$ is a Banach space when equipped with the norm $\norm{f}_{\mathcal{F}_p}=(\norm{f}_{L^p}^p+\mathcal{E}_p(f))^{1/p}$.
			\item \textbf{Homogeneity:} For every $f\in \mathcal{F}_p$ there exists a finite non-negative Borel measure $\Gamma_p \langle f \rangle$ on $X$ such that $\Gamma_p \langle f \rangle(X)= \mathcal{E}_p(f)$, and for all $\lambda \in \mathbb{R}$
			\[
			\Gamma_p\langle \lambda f\rangle = |\lambda|^p \Gamma_p\langle f\rangle.
			\]
			The measure $\Gamma_p\langle f \rangle$ associated to $f \in \mathcal{F}_p$ is called the \emph{energy measure of $f$}.
			\item \textbf{Triangle inequality:} For every $f,g\in \mathcal{F}_p$ and every Borel set $A\subset X$, 
			\begin{equation} \label{e:def-sublinear}
				\Gamma_p\langle f+g\rangle(A)^{\frac{1}{p}}\le \Gamma_p\langle f\rangle(A)^{\frac{1}{p}}+\Gamma_p\langle g\rangle(A)^{\frac{1}{p}}
			\end{equation}
			\item \textbf{Lipschitz contractivity:} For all $f\in \mathcal{F}_p$, $g\in \Lip(\mathbb{R})$ with $g(0)=0$,  we have $g \circ f \in \mathcal{F}_p$ and
			\begin{equation} \label{e:def-chain}
				d\Gamma_p\langle g\circ f\rangle \le \LIP[g]^p d\Gamma_p\langle f\rangle.
			\end{equation}
			\item \textbf{Strong locality:} For every $f\in \mathcal{F}_p$ and any open set $A\subset X$, if $f|_A=c$ $\mu$-a.e., then $\Gamma_p\langle f\rangle(A)=0$. 
			\item \textbf{Weak lower semicontinuity:} For every $f\in L^p(X)$, and any sequence of functions $(f_i)_{i \in \mathbb{N}}$ in $\mathcal{F}_p$ such that $f_i\to f \in L^p(X)$ with $\sup_{i\in \mathbb{N}} \mathcal{E}_p(f_i)<\infty$, then $f\in \mathcal{F}_p$ and
			\begin{equation} \label{e:def-lsc}
				\Gamma_p\langle f \rangle(A)\le \liminf_{i\to \infty} \Gamma_p\langle f_i \rangle(A)
			\end{equation}
			for every Borel set $A\subset X$. 
			\item \textbf{Regularity:} $C_c(X)\cap \cF_p$ is dense in $\cF_p$ in the norm $\norm{f}_{\mathcal{F}_p}$ and in $C_c(X)$ in the norm $\|f\|_\infty = \sup_{x\in X} |f(x)|.$
		\end{enumerate}
	\end{definition}

\vskip2em

\begin{center}
\noindent \emph{\textbf{Standing assumption:} We will assume throughout the rest of the paper that $(X,d,\mu, \mathcal{E}_p, \mathcal{F}_p, \Gamma_p)$ is a regular local $p$-Dirichlet structure.}
\end{center}

\vskip2em

If $O\subset X$ is an open set, we define its capacity as $\cCap(O)=\inf\{ \norm{f}_{\mathcal{F}_p}^p : f\in \cF_p , f|_O = 1\}$. For $A\subset X$, we define $\cCap(A)=\inf\{\cCap(O) : A\subset O,\  O \text{ is open }\}.$

A function $f$ is called quasicontinuous, if for every $\epsilon>0$ there exists an open set $O$ with $\cCap(O)<\epsilon$ s.t. $f|_{X\setminus O}$ is continuous. If $f_i:X\to [-\infty,\infty]$ is a sequence of functions, we say that $f_i\to f$ quasieverywhere if for every $\epsilon>0$ there exists an open set $O$ with $\cCap(O)<\epsilon$ s.t. $\lim_{i\to \infty} f_i(x)=f(x)$ for $x\not\in O$. 

The space $\cF_p$, as we define it,  consists of equivalence classes, and each equivalence class contains an essentially unique quasicontinuous representative: We call $\tilde{f}:X\to [-\infty,\infty]$ a quasicontinuous representative for $f\in \cF_p$ if $\tilde{f}$ is measurable, quasicontinuous and $\tilde{f}=f$ a.e. In the following proposition, we collect this and some other basic properties of $p$-Dirichlet spaces from \cite{erikssonmurugan}.

\begin{proposition}\label{prop:properties} 
\begin{enumerate}
    \item[(i)] $\mu(O)\leq \cCap(O)$ for all Borel sets $O\subset X$.
    \item[(ii)] For every $f\in \cF_p$ there exists a quasicontinuous representative.
    \item[(iii)] If $\tilde{f}, \tilde{\tilde{f}}$ are two quasicontinuous representatives of $f$, then $\cCap(\{x:\tilde{f}(x)\neq \tilde{\tilde{f}}(x)\}=0$.
    \item[(iv)] For every $f\in \cF_p$ and every set $A$ with $\cCap(A)=0$ we have $\Gamma_p\langle f\rangle(A)=0$.
    \item[(v)] If $f,g\in \cF_p$ are quasicontinuous, $A\subset X$ is measurable and $f|_A=g|_A$, then $\Gamma_p\langle f\rangle(A)=\Gamma_p\langle g\rangle(A)$.
    \item[(vi)] If $\phi\in \Lip_0(\R)$ and $f\in \cF_p$ is quasicontinuous, then $\phi\circ f\in \cF_p$ with $\Gamma_p\langle f\rangle \leq \Lip_a[\phi] \circ f \Gamma_p\langle f\rangle$.
    \item[(vii)] If $f_i\in C_c(X)\cap \cF_p$ is a sequence of continuous functions converging to $f$ in $\cF_p$, then a subsequence converges quasi-everywhere to a quasicontinuous representative $\tilde{f}$ of $f$.
    \item[(viii)] If $\phi \in \Lip_0(\R^n)$ is continuously differentiable, and $f_1,\dots, f_n \in \cF_p$ are quasicontinuous then $\phi(f_1,\dots, f_n)\in \cF$ and
    \[
    \Gamma_p\langle \phi(f_1,\dots, f_n)\rangle \leq 1 \vee n^{\frac{p-2}{2}}\sum_{i=1}^n |\partial_i \phi(f_1,\dots, f_n)|^p \Gamma_p\langle f_i\rangle
    \]
    
\end{enumerate}
\end{proposition}
\begin{proof}
    The first property is immediate from the definition for open sets, and follows from Borel regularity of $\mu$ for all Borel sets. Properties (ii)-(vi) follow from \cite[Proposition 2.34 and Proposition 2.35]{erikssonmurugan}. 

    Claim (vii) follows by a classical argument, see e.g. \cite[Proof of Theorem  3.7]{Shanmun}. Pass to a subsequence s.t. $\|f_i-f_j\|_{\cF_p}^p\leq 4^{-\min(i,j)}$. Let $E_{i,j} = \{ x: |f_i-f_j|> 2^{-i/p}\}$. Considering the function $h=\min(2^{i/p}|f_i-f_j|,1)$, one sees that $\cCap(E_{i,j}) \leq 2^{-i}$. Further, let $E_i=\bigcup_{j\geq i} E_{i,j}$ and $E=\bigcap_{i\in \N} E_i$. By \cite[Lemma 2.27]{erikssonmurugan} $\cCap(E_i)\leq 2^{p-1}\left(\sum_{j=i}^\infty 2^{-i/p}\right)^p$, and thus $\cCap(E)=0$. It is direct to verify that $f_i$ is a Cauchy sequence for $x\not\in E$. Let $\tilde{f}(x)=\lim_{i\to \infty} f_i(x)$, which is now defined for all $x\not\in E$. For $x\in E$, set $\tilde{f}(x)=0$. Since $f_i\to f$ in $L^p(X)$, we get $\tilde{f}=f$ $\mu$-a.e. The function $\tilde{f}$ is a uniform limit of continuous functions outside the open set $E_i$, and thus $\tilde{f}$ is continuous outside $E_i$. Thus $\tilde{f}$ is a quasicontinuous representative for $f$.

    Claim (viii) has a proof quite similar to \cite[Proposition 2.35(3)]{erikssonmurugan}. First, consider a linear function $\phi(x)=\sum_{i=1}^n\lambda_i x_i$. By the triangle inequality and locality, we have
    \[
    \Gamma_p\langle \phi( f_1,\dots, f_n)\rangle (A)^\frac{1}{p} \leq \sum_{i=1}^n |\partial_i \phi(f_1,\dots, f_n)|\Gamma_p\langle f_i\rangle (A)^\frac{1}{p}.
    \]
    Using H\"older's inequality, we get
    \[
    \Gamma_p\langle \phi(f_1,\dots, f_n)\rangle (A) \leq 1\vee n^{\frac{p-2}{2}}\sum_{i=1}^n |\partial_i \phi(f_1,\dots, f_n)|^p\Gamma_p\langle f_i\rangle (A).
    \]
    This is the desired claim for linear functions. 
    
    Next, let $\epsilon>0$ and divide $\R^n$ to simplices $\Delta(p_1,\dots, p_{n+1}):=\{\sum_{i=1}^{n+1} w_i p_i : \sum_{i=1}^n w_i =1, w_i \geq 0\}$ for $(p_1,\dots, p_{n+1})\in S_\epsilon$, in such a way that the following properties hold:
    \begin{enumerate}
        \item The  simplicies form a simplicial complex, i.e. distinct simplices have disjoint interiors and so that pairs of simplices intersect in their faces.
        \item The simplicies are othogonal: $p_i-p_1=\xi_i \epsilon e_{i-1}$, where $e_1,\dots, e_n$ is a standard basis and $\xi_i=\pm1$ for $i=1,\dots, n$. 
    \end{enumerate}
    Then   $\diam(\Delta(p_1,\dots, p_{n+1}))\leq 2\epsilon$. Such simplices can be generated by taking the tesselation corresponding to the fundamental domain $\Delta_0 = \{(x_1,\dots, x_n) : 0\leq x_1 \leq x_2 \cdots \leq x_n \leq \epsilon\}$ of the group $G$ of isometries on $\R^n$ generated by all translations by vectors in $\epsilon \Z^n$, and permutations of co-ordinates.
 
    For each $i=2,\dots, {n+1}$, there exists $q_i\in\Delta(p_1,\dots, p_{n+1})$ s.t. $\phi(p_i)-\phi(p_1)=\langle\nabla \phi(q_i),p_i-p_1\rangle$. Since $p_i-p_1$ are orthogonal, there exists a vector $F(p_1,\dots, p_{n+1})$ s.t. 
    \[
    \langle F(p_1,\dots, p_{n+1}), p_i-p_1\rangle = \langle\nabla \phi(q_i),p_i-p_1\rangle,
    \]
    and the $i$'th component of $F(p_1,\dots, p_n)$ is given by the $i$'th component of $\nabla \phi(q_i)$. This allows us to define approximations $\phi_\epsilon(x)=\sum_{i=1}^{n+1} w_i \phi(p_i)$ for $x=\sum_{i=1}^{n+1} w_i p_i\in \Delta(p_1,\dots, p_{n+1})$ and $(p_1,\dots, p_{n+1})\in S_\epsilon$. The functions $\phi_\epsilon$ are uniformly Lipschitz, and piecewise linear. By a simple application of the mean-value property, for $x\in \Delta(p_1,\dots, p_{n+1}))$, we have
    \begin{align*}
    \phi_\epsilon(x)&=\phi(p_1)+\sum_{i=1}^{n+1} w_i (\phi(p_i)-\phi(p_1)) \\
    &= \phi(p_1)+\sum_{i=1}^{n+1} w_i \langle\nabla \phi(q_i),p_i-p_1\rangle  \\
    &= \phi(p_1)+\sum_{i=1}^{n+1} w_i \langle F(p_1,\dots, p_{n+1}),p_i-p_1\rangle  \\
    &=\phi(p_1)+\langle F(p_1,\dots, p_{n+1}),\sum_{i=1}^{n+1}w_ip_i-p_1\rangle \\
    &=\phi(p_1)+\langle F(p_1,\dots, p_{n+1}),x-p_1\rangle.
    \end{align*}
    Thus, 
    \[
    \Gamma_p\langle \phi_\epsilon \circ f\rangle|_{f^{-1}(\Delta(p_1,\dots, p_{n+1}))} \leq 1 \vee n^{\frac{p-2}{2}} \sum_{i=1}^n |F(p_1,\dots, p_{n+1})_i|^p \Gamma_p\langle f_i\rangle|_{f^{-1}(\Delta(p_1,\dots, p_{n+1}))}.
    \]
    Now, choose for each $x\in X$ a $\mathbf{p}_{\epsilon}(x)=(p_1(x),\dots, p_{n+1}(x))\in S_\epsilon$, s.t. $(f_1(x), \dots, f_n(x))\in \Delta(\mathbf{p}_{\epsilon}(x))$. There is an ambiguity whenever $(f_1(x), \dots, f_n(x))$ lies on a face of a simplex, but any choise that guarantees measurability will do (e.g. for each of the countably many faces, choose one side).  It is direct to see that $F(\mathbf{p}_{\epsilon}(x))\to \nabla\phi(f_1,\dots, f_n)(x)$ pointwise everywhere. By using this on the inequality
    \[
    \Gamma_p\langle \phi_\epsilon \circ f\rangle \leq 1 \vee n^{\frac{p-2}{2}} \sum_{i=1}^n F(\mathbf{p}_{\epsilon}(x))_i \Gamma_p\langle f_i\rangle,
    \]
    the claim follows by sending $\epsilon \to 0$.
\end{proof}

Finally, we give a reduction of the general cutoff Sobolev inequality to the case of continuous functions.

\begin{lemma}\label{lem:reductiontocont}
    If the cutoff Sobolev inequality holds for all $f\in C_c(X)\cap \cF_p$, then it holds for all  $f\in \cF_p$.
\end{lemma}
\begin{proof}
    Let $f\in \cF_p$ and assume by passing to the quasicontinuous representative that $f$ is quasicontinuous. 
    
    By regularity, there exists a sequence of continuous functions  $f_i\in C_c(X)\cap \cF_p$ converging to $f$ in $\cF_p$. By Proposition \ref{prop:properties}, we have after passing to a subsequence, that $f_i$ converges to $f$ quasi-everywhere. Thus, $f_i(x)\to f(x)$ holds $\Gamma_p\langle \psi_B\rangle$-a.e. The cutoff Sobolev inequality for $f$ now follows from Fatou's Lemma.
\end{proof}

\subsection{Harmonic and superharmonic functions}

Let $\Omega\subset X$ be a non-empty open set. We let $\cF_{p,\loc}(\Omega)$ to be those functions $u\in L^p(\Omega)$ s.t. for every open set $\Omega'\subset \Omega$ with compact closure there exists a $\tilde{u}\in \cF_p$ s.t. $\tilde{u}|_{\Omega'} = u$. If there are two such extensions $\tilde{u}_1,\tilde{u}_2$, then $\Gamma_p\langle \tilde{u}_1\rangle|_{\Omega'}  = \Gamma_p\langle \tilde{u}_2\rangle|_{\Omega'}$ by locality. Thus, if $A\subset \Omega$ is a compact set, we can define $\Gamma_p\langle u\rangle(A):=\Gamma_p\langle \tilde{u}\rangle(A)$ where first we choose an open set $\Omega'\subset \Omega$ with compact closure s.t. $A\subset \Omega'$, and then choose an arbitrary extension as before.  One directly sees that if $u\in \cF_{p,{\rm loc}}(\Omega)$ and $v\in \cF_{p,{\rm loc}}(\Omega)$ then $u+v\in \cF_p(\Omega)$. It is easy to verify that this construction satisfies the following properties.

\begin{lemma}\label{lem:localenergy} If $\Omega\subset X$ is open, then the following hold.
    \begin{enumerate}  
			\item \textbf{Homogeneity:} For every $f\in \mathcal{F}_{p,\loc}(\Omega)$ there exists a finite non-negative Borel measure $\Gamma_p \langle f \rangle$ on $\Omega$ and for all $\lambda \in \mathbb{R}$
			\[
			\Gamma_p\langle \lambda f\rangle = |\lambda|^p \Gamma_p\langle f\rangle.
			\]
			\item \textbf{Triangle inequality:} For every $f,g\in \mathcal{F}_{p,\loc}(\Omega)$ and every Borel set $A\subset \Omega$, 
			\begin{equation} \label{e:def-sublinear-loc}
				\Gamma_p\langle f+g\rangle(A)^{\frac{1}{p}}\le \Gamma_p\langle f\rangle(A)^{\frac{1}{p}}+\Gamma_p\langle g\rangle(A)^{\frac{1}{p}}
			\end{equation}
            \item \textbf{Addition of constant:} For all $f\in \mathcal{F}_{p,\loc}(\Omega)$ and $c\in \R$ we have $f+c\in \mathcal{F}_{p,\loc}(\Omega)$ and $\Gamma_p\langle f\rangle=\Gamma_p\langle f+c\rangle$.
			\item \textbf{Lipschitz contractivity:} For all $f\in \mathcal{F}_{p,\loc}(\Omega)$ and $g\in \Lip(\mathbb{R})$,  we have $g \circ f \in \mathcal{F}_{p,\loc}(\Omega)$ and
			\begin{equation} \label{e:def-chain-local}
				d\Gamma_p\langle g\circ f\rangle \le \LIP[g]^p d\Gamma_p\langle f\rangle.
			\end{equation}
			\item \textbf{Strong locality:} For every $f\in \mathcal{F}_{p,\loc}(\Omega)$ and any open set $A\subset \Omega$, if $f|_\Omega=c$ $\mu$-a.e., then $\Gamma_p\langle f\rangle(A)=0$. 
		\end{enumerate}
\end{lemma}
\begin{proof}
    The homogeneity, triangle inequalities and strong locality follow directly from the axioms, since we can apply them to the extensions of $f,g$. 

    Next, we show that $f+c\in \mathcal{F}_{p,\loc}(\Omega)$ for every $c\in \R$ with $\Gamma_p\langle f+c\rangle=\Gamma_p\langle f\rangle$. Indeed, let $\Omega'\subset \Omega$ be an arbitrary open set with compact closure and $\widetilde{f}$ the associated extension of $f$ to $\cF_p$. By regularity and Uryhsons lemma there exists a continuous function $g\in \cF_p$ with $g\geq c$ on $\Omega'$. Let $h=\min(g,c)$, and then let $\widetilde{f+c}=\widetilde{f}+h \in \cF_p$. Moreover, $\Gamma_p\langle h\rangle|_{\Omega'}=0$ by locality, so $\Gamma_p\langle f+c\rangle = \Gamma_p\langle \widetilde{f+c}\rangle|_\Omega = \Gamma_p\langle \widetilde{f}\rangle|_\Omega=\Gamma_p\langle f\rangle$.

    Now Lipschitz contractivity follows since we can apply the usual contractivity with the function $\tilde{g}(x)=g(x)-g(0)$ together with the addition of a constant-property.
\end{proof}

Let $\spt(\psi)=\overline{\{x \in X: \psi(x)>0\}}$ denote the support of $\psi$. Let $\cF_{p,0}(\Omega)$ to be the collection of those $u\in \cF_{p, {\rm loc}}(\Omega)$ for which the extension $\tilde{u}$ defined by $\tilde{u}|_\Omega = u$ and $\tilde{u}|_{X\setminus \Omega}=0$ satisfies $\tilde{u}\in \cF_p$. 

Let $\mathcal{H}(\Omega)$ be the collection of harmonic functions: i.e  those functions $u\in \cF_{p, {\rm loc}}(\Omega)$ s.t. \[\Gamma_p\langle u+\psi\rangle(\Omega \cap \spt(\psi))\geq \Gamma_p\langle u\rangle(\Omega \cap \spt(\psi))\] for every $\psi \in \cF_{p,0}(\Omega)$ with $\psi|_{X\setminus \Omega}=0$.
Let $\Omega\subset X$ be a non-empty open set and let $\mathcal{H}_+(\Omega)$ be the collection of super-harmonic functions: i.e  those functions $u\in \cF_p(\Omega)$  s.t. \[\Gamma_p\langle u+\psi\rangle(\Omega \cap \spt(\psi))\geq \Gamma_p\langle u\rangle(\Omega \cap \spt(\psi))\] for every $\psi \in \cF_p$ with $\psi|_{X\setminus \Omega}=0$, $\psi\geq 0$.

\begin{lemma}\label{lem:existenceofaminimizer} Assume that $X$ is proper.  For every $B\subset X$, there exists a function $\phi_B:X\to [0,1]$ with $\phi_B\in \cF_p$, s.t. $\phi_B|_B=1$,  $\phi_B|_{X\setminus 2B}=0$ and $\phi_B \in \mathcal{H}(2B\setminus B), \phi_B \in \mathcal{H}_+(2B)$ and $1-\phi_B\in \mathcal{H}_+(X\setminus B)$. 
\end{lemma}
\begin{proof} Let $\mathcal{G}=\{\phi\in \cF_p : \phi:X\to [0,1], \ \phi|_B=1,  \ \phi|_{X\setminus 2B}=0\}$. First, we argue that this collection is non-empty. By Uryhson's lemma, there exists a continuous function $\psi\in C_c(X)$ with $\phi:X\to [0,1], \phi|_B=1,  \phi|_{X\setminus 2B}=0$. By regularity, there exists a function $\xi\in C_c(X)\cap \cF_p$ such that $\|\xi-\psi\|_\infty\leq 1/4$. Let $\phi = \frac{4}{3}\min(\max(\xi-1/4, 0),3/4)$. It is direct to verify with Lipchitz contractivity that $\phi\in \mathcal{G}$.

Let $E = \inf\{\cE_p(\phi) : \phi \in \mathcal{G}\}$ and let $\phi_i\in \mathcal{G}$, for $i\in \N$,  be functions s.t. $\cE_p(\phi_i)\searrow E$. We have $\phi_i\in L^p(X)$ and $\|\phi_i\|_{L^p}\leq \mu(2B)^\frac{1}{p}.$ Thus, by Mazur's lemma, we can find convex combinations $\psi_i = \sum_{k=i}^{n_i} \alpha_{i,k} \phi_i$ s.t. $\psi_i$ converges in $L^p(X)$, for some $n_i\geq i$ and $\alpha_{i,k}\in [0,1]$ with $\sum_{k=i}^{n_i} \alpha_{i,k}=1$. By the triangle inequality and convexity, we have $\psi_i\in \mathcal{G}$ and $\cE_p(\psi_i)\searrow E$. 

The $L^p(X)$-limit $u$ of $\psi_i$ satisfies $u:X\to [0,1], u|_B=1$ and $u|_{X\setminus 2B}=0$. Further, by weak lower semicontinuity, we have $\phi\in \cF$ and thus $\phi \in \mathcal{G}$ with $\cE_p(u)\leq E$. Since $E$ was the infimum, we must have $\cE_p(u)=E$. Let $\phi_B=\max(\min(u,1),0)$. It is direct to see that $\phi_B\in \mathcal{G}$ and $\cE_p(\phi_B)\leq \cE_p(u)$. Thus, $\cE_p(\phi)=E$ is also a minimizer.

Next, if $\psi\in \cF_p$ satisfies $\psi|_{B\cup X \setminus 2B} = 0$, then $\phi + \psi \in \mathcal{G}$, and thus 
\[
\Gamma_p\langle \phi+\psi\rangle(\Omega) = \cE_p(\phi+\psi)\geq E = \cE_p(\phi)=\Gamma_p\langle \phi\rangle(\Omega),
\]
where the first and third follow from locality, since $\phi+\psi=\phi=1$ on $B$ and $\phi+\psi=\phi=0$ on $X\setminus 2B$. Since $\phi+\psi = \phi$ outside the closed set $\spt(\psi)$, we get by locality that 
\[
\Gamma_p\langle \phi+\psi\rangle(\Omega \cap \spt(\psi)) \geq \Gamma_p\langle \phi\rangle(\Omega \cap \spt(\psi))
\]
Thus $\phi_B\in \cH(2B\setminus B)$.

Next, if  $\psi\in \cF_p$ satisfies $\psi|_{X \setminus 2B} = 0$ and $\psi\geq 0$, then $\min(\phi + \psi,1) \in \mathcal{G}$, and thus 
\[
\Gamma_p\langle \phi+\psi\rangle(2B) \geq \cE_p(\min(\phi + \psi,1))\geq E=\cE_p(\phi)=\Gamma_p\langle \phi\rangle(2B).
\]
Thus, by the same locality argument as above, $\phi \in \cH_+(2B)$.

Finally, $1-\phi_B\in \cF_{p, \rm loc}(X\setminus B)$ by Lemma \ref{lem:localenergy}. If $\psi\in \cF_p$ satisfies $\psi|_{B} = 0$ and $\psi\geq 0$, then $\max(\phi - \psi,0) \in \mathcal{G}$ and
\[
\Gamma_p\langle \phi-\psi\rangle(2B\setminus B) \geq \cE_p(\max(\phi - \psi,0))\geq \cE_p(\phi)=\Gamma_p\langle \phi\rangle(X\setminus B)=\Gamma_p\langle 1-\phi\rangle(X\setminus B).
\]
By sub-additivity and locality and Lemma \ref{lem:localenergy}, we get
\[
\Gamma_p\langle 1-\phi+\psi\rangle(2B\setminus B \cap \spt(\psi))=\Gamma_p\langle 1-\phi\rangle(2B\setminus B \cap \spt(\psi)).
\]
Thus, $1-\psi_B \in \cH_+(X\setminus B)$
\end{proof}

%\begin{lemma}\label{lem:sheafprop}
%    Let $X=A\cup B$ for some open sets $A,B\subset X$, and assume that $f,g\in \cF_p$. If $h:X\to \R$ satisfies $h|_A=f$ and $h|_B=g$, then $h\in \cF_p$. 
%\end{lemma}
%\begin{proof}
% If needed
%\end{proof}

\subsection{Poincar\'e inequality and capacity bounds}
Following \cite{barlow2020stability}, we call $\Psi:X\times [0,\infty)\to (0,\infty)$ a regular scale function if $\Psi(x,0)=0$ and there exists a constant $C_{\rm \Psi}>0$ and exponents $\beta_{\pm}>0$ s.t. for all $x,y\in X$ and all $0<s\leq r \leq 2\diam(X)$, $R=d(x,y)$ we have
\[
C_{\rm \Psi}^{-1} \left(\frac{s \vee R}{r}\right)^{\beta_{+}} \left(\frac{s}{r \vee R}\right)^{\beta_{-}} \leq \frac{\Psi(x,r)}{\Psi(y,s)}\leq C_{\rm \Psi}\left(\frac{r}{r \vee R}\right)^{\beta_{-}}\left(\frac{r \vee R}{s}\right)^{\beta_{+}}.
\]
Where convenient, we will abbreviate $\Psi(B):=\Psi(x_B,\rad(B))$.

Often, we will assume one of two properties:

\begin{enumerate}
    \item[$(\cCap_p(\Psi))$] Capacity upper bound: For every ball $B\subset X$ with $\rad(B)\leq 2\diam(X)$ there exists a $\psi_B \in C(X)\cap \cF_p$ with $\psi_B|B=1$ and $\psi_B|_{X\setminus 2B}=0$ and
    \[
    \Gamma_p\langle \psi_B \rangle(X)=\Gamma_p\langle \psi_B \rangle(2B)\leq C_{\rm cap}\frac{\mu(B)}{\Psi(B)} 
    \]
    \item[$(\PI_p(\Psi))$] Poincar\'e inequality: For every $B=B(x,r)\subset X$ with $r\leq 2\diam(X)$ and every $f\in \cF_p$
    \[
    \fint_B |f - f_B|^p d\mu \leq C_{\rm PI}\frac{\Psi(B)}{\mu(B)}\Gamma_p\langle f\rangle(\Lambda B). 
    \]
\end{enumerate}
Without loss of generality, we may assume $\Lambda \geq 8$ by possibly increasing the constant.

\begin{remark}\label{rmk:local} The Poincar\'e inequality is expressed as a condition for functions $f\in \cF_p$. It also holds for functions $f\in \cF_{p, \loc}(2\Lambda B)$ when $\mu$ is doubling by the following short argument. By doubling,  $\overline{\Lambda B}$ is compact. If $f\in \cF_p(2\Lambda B)$, then there is a $\tilde{f}\in \cF_p$ s.t. $\tilde{f}|_{\Lambda B}=f|_{\Lambda B}.$ Applying the Poincar\'e inequality to $\tilde{f}$ yields the same for $f$. 
\end{remark}

In proving the Poincar\'e inequality, the following Lemma will be frequently useful.

\begin{lemma} \label{lem:doublingballs} Assume that $\mu$ is $C_D-$doubling. For every $L\geq 1$ there exists a constant $C_L=C_L(C_D)$ s.t. if $B' \subset LB$ and $\rad(B')\geq 1/L \rad(B)$, then
\[
|f_B' -f_B|^p \leq C_L \fint_{LB} |f-f_B|^p d\mu \leq 2^p C_L^2 \fint_{LB} |f-f_{LB}|^p d\mu.
\]
\end{lemma}
\begin{proof} By Jensen's inequality
\[
|f_B' -f_B|^p \leq \frac{1}{\mu(B')}\int_{B'} |f-f_B|^p d\mu.
\]
The first inequality then follows by increasing the domain and using doubling. The second inequality follows by
\[
|f-f_B|\leq |f-f_{LB}| + |f_B-f_{LB}|,
\]
and again applying Jensen's inequality.
\end{proof} 

As noted in the introduction, there is a different form of the cutoff Sobolev inequality that has appeared in the literature. We say that a $p$-Dirichlet structure satisfies the classical cutoff Sobolev inequality ${\rm CS}_{\rm classical}(\Psi)$ if the following holds: there exists a constants $C,\Lambda \geq 1$ s.t. for every $B(x,r)$ there exists a $\phi_B \in \cF_p$ with $\phi_B|B=1$ and $\phi_B|_{X\setminus 2B}=0$ s.t. for every $f\in \cF_p$ and every quasicontinuous representative $\tilde{f}$ of $f$ we have
\[
\fint_{2B} |\tilde{f}|^p d\Gamma_p\langle \phi_B\rangle \leq C \Gamma_p\langle f\rangle(\Lambda B) + C\frac{1}{\Psi(B)}\int_{2B} f^p d\mu. 
\]
In order to use the results of \cite{barlow2006stability,GHL2,barlowanders,MN}, we will quickly prove the equivalence between this and the cutoff Sobolev inequality that we use.
\begin{lemma}\label{lem:equivalence}
    Assume that the $p$-Dirichlet structure stucture satisfies volume doubling, $\cCap_p(\Psi)$ and $\PI_p(\Psi)$. Then 
    \[{\rm CS}_{\rm classical}(\Psi) \Longleftrightarrow {\rm CS}_p\]
\end{lemma}
\begin{proof}
    First, assume ${\rm CS}_{\rm classical}(\Psi)$. Setting $g=f-f_B$, then ${\rm CS}_{\rm classical}(\Psi)$ and $\PI_p(\Psi)$ imply:
    \begin{align*}
        \fint_{2B} |\tilde{f}-f_B|^p d\Gamma_p\langle \phi_B\rangle &\leq C \Gamma_p\langle f\rangle(\Lambda B) + C\frac{1}{\Psi(B)}\int_{2B} |f-f_B|^p d\mu \\
        &\leq C \Gamma_p\langle f\rangle(\Lambda B) + CC_{\rm PI}\Gamma_p\langle f\rangle(\Lambda B) \\
        &\lesssim \Gamma_p\langle f\rangle(\Lambda B).
    \end{align*}
    Here, we applied the Poincar\'e inequality via Remark \ref{rmk:local} to $\tilde{f}-f_B$ for which we have $\tilde{f}-f_B\in \cF_{p,\loc}(X)$ by Lemma \ref{lem:localenergy}.

    Next, assume ${\rm CS}_p$. Then $\cCap_p(\Psi)$ implies:
    \begin{align*}
        \fint_{2B} |\tilde{f}|^p d\Gamma_p\langle \phi_B\rangle &\lesssim \fint_{2B} |\tilde{f}-f_B|^p d\Gamma_p\langle \phi_B\rangle  + \Gamma_p\langle \phi_B\rangle(2B) |f_B|^p \\
        &\lesssim 
        C_{\rm CS} \Gamma_p\langle f\rangle(\Lambda B) + \frac{1}{\Psi(B)}\int_{2B} |f|^p d\mu.
    \end{align*}
\end{proof}

There are also many other forms of the cutoff Sobolev inequality, see the introduction and \cite{barlowanders,GHL, GHL2,rikuanttila,yangenergy} for some versions and their equivalences. 

\section{Proof of Main Theorems}\label{sec:proofofmain}

\subsection{Proof of Theorem \ref{thm:mainthm}}
 
The key estimate is the following log-Caccioppoli-type inequality. This terminology comes from the fact that in the classical case, up to a constant, it can be shown by an application of a log-Caccioppoli inequality. 

\begin{lemma}\label{lem:LogCaccioppoli}  
    If $u\in \mathcal{H}_+(\Omega)$ is a quasicontinuous superharmonic function with $\Gamma_p\langle u\rangle(\Omega)<\infty$ in $\Omega$ and $0\leq u \leq 1$, and if $h\in \cF_p$ is quasicontinous, $h|_{X\setminus \Omega}\leq 0$ and $A\subset \Omega$ are such that $h|_A \geq 1$, then
    \[
    \Gamma_p\langle u\rangle(A) \leq \Gamma_p\langle h\rangle(\Omega).
    \]
\end{lemma}
\begin{proof}

    Let $\tilde{u} = \min(\max(u,h),1)$, then $\tilde{u}-u\in \cF_{p,0}(\Omega)$ and $\tilde{u}-u\geq 0$ and $\spt(\tilde{u}-u)\subset \spt(h)$. 
    \begin{equation*}
    \Gamma_p\langle \tilde{u}\rangle(\Omega \cap \spt(\tilde{u}-u))\geq  \Gamma_p\langle u\rangle(\Omega \cap \spt(\tilde{u}-u)).
     \end{equation*}
     Since by locality $\Gamma_p\langle u\rangle(\Omega \setminus \spt(\tilde{u}-u))=\Gamma_p\langle \tilde{u}\rangle(\Omega \setminus \spt(\tilde{u}-u))$, we get
    \begin{equation} \label{eq:harmonicityofu}
    \Gamma_p\langle \tilde{u}\rangle(\Omega)\geq  \Gamma_p\langle u\rangle(\Omega).
     \end{equation}

    Additionally, by Proposition \ref{prop:properties}, $\tilde{u}|_A=1$ and quasicontinuity, we get 
    \[
    \Gamma_p\langle \tilde{u}\rangle(A)=0.
    \]
    Now,
    \[
    \Gamma_p\langle \tilde{u}\rangle(\Omega)\leq \Gamma_p\langle u\rangle(\Omega\setminus A)+ \Gamma_p\langle h\rangle(\Omega\setminus A),
    \]
    and
    \[
    \Gamma_p\langle u\rangle(\Omega)=\Gamma_p\langle u\rangle(\Omega \setminus A)+\Gamma_p\langle u\rangle(A).
    \]
    Combining these with \eqref{eq:harmonicityofu}, we get
    \[
    \Gamma_p\langle u\rangle(A)\leq \Gamma_p\langle h\rangle(\Omega).
    \]
\end{proof}

This inequality suffices to prove the cutoff Sobolev inequality for functions $f$ that vanish outside $2B$ or in the ball $B$. In order to reduce to this case, we need to be able to blend functions together, where we use Lemma \ref{lem:whitneyblending}. We now show how a more general version of the main theorem follows from this with a fairly short proof. Notice that Theorem \ref{thm:mainthm} follows from this version by choosing $\Psi(x,r)=r^\beta$. A reader may wish to compare the proof to the proofs of \cite[Lemma 3.3  and Theorem 3.5]{kinnunenkorte}.

\begin{theorem}\label{thm:mainthm-sharper} Assume that $(X,d,\mu,\cF_p,\Gamma_p)$ is a regular local $p$-Dirichlet structure, with $\mu$ doubling, and satisfying $\cCap_p(\Psi)$ and $\PI_p(\Psi)$. Then, $(X,d,\mu, \cF_p,\Gamma_p)$ satisfies the $p$-cutoff Sobolev inequality ${\rm CS}_p$.  
\end{theorem}

\begin{proof}
    Let $\psi_B$ be the harmonic functions on $2B\setminus B$, which equal $1$ on $B$ and $0$ outside $2B$. These are guaranteed to exist by Lemma \ref{lem:existenceofaminimizer}. By Lemma \ref{lem:reductiontocont} it suffices to assume that $f$ is continuous.  By subtracting $f_{2B}\psi_{2\Lambda B}$, where $\psi_{2\Lambda B}|_{2\Lambda B}=1$, we may assume that $f_{2B}=0$.

    Construct two functions using Lemma \ref{lem:whitneyblending}: $h_1$ s.t. $h_1|_B=0$ and $h_1|_{X\setminus 3/2 B} = f|_{X\setminus 3/2B}$, and $h_2$ s.t. $h_2|_{X\setminus 2B}=0$ and $h_2|_{3/2B}=f$.  We may assume that $h_i$ are quasicontinuous by adjusting representatives. 
    We have by Lemma \ref{lem:whitneyblending} combined with the Poincar\'e inequality:
    \[
    \Gamma_p\langle h_i\rangle(\Lambda B)\lesssim \Gamma_p\langle f\rangle(\Lambda B).
    \]

    Further,
    \[
    |f|^p\leq |h_1|^p + |h_2|^p
    \]
    for every $x$, and thus cutoff Sobolev is true for $f$ if we have the corresponding bound with $f$ replaced with $h_1$ and $h_2$.

    Consider first $i=1$, and take $g=h_1$ and $\Omega = 2B$. By Lemma \ref{lem:existenceofaminimizer} $u=\psi_B$ is superharmonic in $\Omega$. For every $\lambda>0$, let $g_\lambda = \frac{\max(\min(|g|,2\lambda)-\lambda,0)}{\lambda}$ a non-negative function in $\Omega$ with $g\geq 1$ on the set $A_{2\lambda}=\{|h|>2\lambda\}$ and $g=0$ on $X\setminus \Omega$, and we get from Lemma \ref{lem:LogCaccioppoli}, that
    \[
    \Gamma_p\langle \psi_B\rangle(A_{2\lambda}) \leq \Gamma_p\langle g_\lambda\rangle(2B)\leq \frac{1}{\lambda^p} \Gamma_p\langle |g|\rangle(A_\lambda \setminus A_{2\lambda}),
    \]
    where we used Proposition \ref{prop:properties} and Lipschitz contractivity.
    
    Now, summing over $i$, we get
    \begin{align}
    \int |h_1|^p d\Gamma_p\langle\psi_B\rangle &\leq \sum_{i=-\infty}^\infty 2^{(i+1)p} \Gamma_p\langle \psi_B\rangle(A_{2^i}) \nonumber \\
    &\leq \sum_{i=-\infty}^\infty \Gamma_p\langle |g|\rangle(A_{2^{i}} \setminus A_{2^{i+1}}) \leq 2\Gamma_p\langle |h_1|\rangle \lesssim \Gamma_p\langle f\rangle(\Lambda B). \label{eq:h1bound}
    \end{align}
    On the final line we also used Lipschitz contractivity.

    Next, let $i=2$ and take $u=1-\psi_B$, $\Omega=X\setminus B$ and $g=h_2$ in the argument. By Lemma \ref{lem:existenceofaminimizer} we have that $u$ is superharmonic in $X\setminus B$ and $\Gamma_p\langle 1-\psi_B\rangle(X\setminus B)=\Gamma_p\langle\psi_B\rangle(X\setminus B)<\infty$. Now,
    the same argument as above yields
    \[
    \int |h_2|^p d\Gamma_p\langle1-\psi_B\rangle \lesssim \Gamma_p\langle h_2\rangle(\Lambda B) \leq \Gamma_p\langle f\rangle(\Lambda B).
    \]
    Since $\Gamma_p\langle 1-\psi_B\rangle=\Gamma_p\langle\psi_B\rangle$ by Lemma \ref{lem:localenergy}, we get
    \[
    \int |h_2|^p d\Gamma_p\langle\psi_B\rangle \lesssim  \Gamma_p\langle f\rangle(\Lambda B).
    \]
    By summing this estimate with \eqref{eq:h1bound}, we obtain the desired claim.
\end{proof}

\subsection{Proof of Theorem \ref{thm:mainthmhke}}

We briefly explain how from Theorem \ref{thm:mainthm} one deduces Theorem \ref{thm:mainthmhke}.

\begin{proof}[Proof of Theorem \ref{thm:mainthmhke}]
    The equivalence of $(1),(2),(3)$ with volume doubling, ${\rm PI}(\beta)$ and the cutoff Sobolev inequality is stated in \cite[Theorem 4.5]{MN}. The proof there cites many earlier works, and we refer the reader to details there. The cutoff Sobolev inequality implies the upper capacity bound by applying the cutoff Sobolev inequality to the function $f=1$ (localized as in Remark \ref{rmk:local}). Therefore, the equivalent conditions $(1-3)$ imply $(4)$. Theorem \ref{thm:mainthm} shows that $(4)$ implies the cutoff Sobolev inequality, and thus also conditions $(1-3)$.
\end{proof}

\section{Whitney Blending}\label{sec:whitneyblending}

In this section, we prove Lemma \ref{lem:whitneyblending}. Our task is really that of extension: $h$ is defined in $B\cup X\setminus (1+\eta)B$, and we must extend $h$ in a way that does not increase the Sobolev norm. This is done by an adaptation of the Jones extension technique \cite{jones}. Jones's theorem applies only to uniform domains in Euclidean spaces, but in our case the relevant domains are not necessarily uniform and lie in a general metric measure space. To fix this, we include adaptations to Jones technique that come from \cite{tuominen,Koskela1998,muruganreflected}. Since the method of extension used here involves a Whitney-cover, and blending two functions together, we call this method \emph{Whitney blending}. It is worth noting that \emph{a posteriori} this method is slightly useless in practice. Indeed, once the cutoff-Sobolev inequality has been proven, Lemma \ref{lem:whitneyblending} has a short proof:

\begin{proof}[Proof of Lemma \ref{lem:whitneyblending} with $\eta=1$ using the cutoff Sobolev inequality] Let $\psi_B$ be the function in the cutoff-Sobolev inequality. Let $h=\psi_B f + (1-\psi_B) g$. By using $\phi_M(a,b,c)=\min(ab + (1-a)c,M)$ in Proposition \ref{prop:properties} together with a simple approximation argument with smooth functions (since $\phi_M$ is not continuously differentiable), we get
\[
\Gamma_p\langle \phi_M(\psi_B,f,g)\rangle \leq 1\vee 3^{\frac{p-2}{2}} \left(\Gamma_p\langle\psi_B\rangle |f-g| + \psi_B \Gamma_p\langle f\rangle + (1-\psi_B)\Gamma_p\langle g\rangle\right).
\]
Sending $M\to  \infty$ and $h=\lim_{M\to\infty} \psi_M$, and using the cutoff Sobolev inequality, we obtain that $h$ is the desired extension.
\end{proof}

However, since our goal is to prove Theorem \ref{thm:mainthm}, we must first find a different way to do the extension that is independent of the cutoff Sobolev inequality: via partitions of unities and an averaging technique. We remark that conceptually the method of Whintey blending appears quite useful, see e.g. \cite{murugannonlocal} for another application of this method to a family conjectures of the resistance type.

\subsection{Sobolev partitions of unity}
Let $\cB$ be a collection of balls. We call $\cB$ $C_N$-locally finite, if for every $B\in \cB$ there are at most $C_N$-many balls $B'\in \cB$ s.t. $B'\cap 2B\neq \emptyset$. For $C>0$ write $C\cB=\{CB: B\in \cB\}$ for the collection of inflated balls. If $\cB$ is a collection of balls, a $C_B$-Sobolev partition for $\cB$ is a collection of functions $\{\phi_B\in \cF_p : B\in \cB\}$ s.t. 
\begin{enumerate}
\item $\phi_B|_{X\setminus 2B}=0$
\item $\sum_{B\in \cB} \phi_B(x)=1$ for $x\in \bigcup_{B\in \cB} B$. 
\item 
\begin{equation}\label{eq:phibenergybound}
\Gamma_p\langle \phi_B\rangle \leq C_B\frac{\mu(B(x,r))}{\Psi(x,r)}
\end{equation}
\end{enumerate}

We construct a partition of unity subordinate to a good cover. This result and method is quite standard, for similar earlier results and methods see \cite[Lemma 6.27]{murugan2023first},  \cite[Lemma 2.5]{muruganlength}, \cite[p. 504]{barlow2006stability}, \cite[p. 235]{kanai} \cite[Lemma 4.5]{muruganreflected}.

\begin{lemma}\label{lem:Sobolevpartofunity} Assume $\cCap_p(\Psi)$.
    For every $C_N\in \N$ there exists a constant $C_B>0$ s.t. the following holds. If $\cB$ is a finite or countable collection of balls that is $C_N$-locally finite, then there is a $C_B=C_B(C_N)$-Sobolev partition for $\cB$.
\end{lemma}

\begin{proof}
    Let $\psi_B\in \cF_p$ be the cutoff function for $B\in \cB$ s.t. $\psi_B|_B=1$, $\psi_B|_{X\setminus 2B}=0$ and $\Gamma_p\langle \psi_{B}\rangle(X)\leq C_{\rm cap} \frac{\mu(B(x,r))}{\Psi(x,r)}$ implied by the upper capacity bound. By Proposition \ref{prop:properties}, we may assume that each $\phi_B$ is quasicontinuous.

    Order the elements in $\cB=\{B_n : n\in I\subset \N\}$, where $I=[0,N]\cap \N$ or $I=\N$ depending on if $\cB$ is finite or countable.

For $n\in I$
\[
\phi_{B_n} = \min\left(\psi_{B_n}, 1-\sum_{i=1}^{n-1} \phi_{B_i}\right).
\]

By \cite[Lemma 2.18]{erikssonmurugan} we have that $\phi_{B_n}\in \cF_p$ for every $n\in I$. It is also easy to see by induction that for all $n\in I$
\begin{equation}\label{eq:sumbound}
0\leq \sum_{i=1}^{n} \phi_{B_n}\leq 1,
\end{equation}
and
\begin{equation}\label{eq:phipsi}
\sum_{i=1}^{n} \phi_{B_n}(x)< 1 \Longrightarrow \phi_{B_k}(x)=\psi_{B_k}(x) \text{ for } k =1,\dots, n.
\end{equation}
It is also direct to see that $\phi_{B_n}$ are quasicontinuous for all $n\in I$.

Further, for $x\notin 2B_n$, we have $\psi_{B_n}(x)=0$, and thus by \eqref{eq:sumbound}, we have $\phi_{B_n}(x)=0$ as well. Further, of $x\in \bigcup_{i=1}^n B_i$, then $\sum_{i=1}^n \phi_{B_i}(x)=1$, since otherwise \eqref{eq:phipsi} would contradict $\psi_{B_i}(x)=1$ for $x\in B_i$ when $i=1,\dots, n$. These give the first and third properties of being a Sobolev partition of unity. 

Next, we give the energy bound. For every $n\in I$, let $E=\{x \in X : \phi_{B_n}(x)=\psi_{B_n}(x)\}$ and let $F=X\setminus E$. Let $F_< = \{x\in F: \sum_{i=1}^{n-1} \phi_{B_n}(x)< 1\}$ and $F_= =\{x\in F: \sum_{i=1}^{n-1} \phi_{B_n}(x)= 1\}$. These sets decompose the space, and thus:
\[
\Gamma_p\langle \phi_{B_n}\rangle(X) = \Gamma_p\langle \phi_{B_n}\rangle(E) + \Gamma_p\langle \phi_{B_n}\rangle(F_<) + \Gamma_p\langle \phi_{B_n}\rangle(F_=).
\]

For $x\in F_=$, we have $\phi_{B_n}(x)=0$, and thus by Proposition \ref{prop:properties} we have
\[
\Gamma_p\langle \phi_{B_n}\rangle(F_=) = 0.
\]
Next, for $x\in E$, we have $\phi_{B_n}(x)=\psi_{B_n}(x)$, and thus
\[
\Gamma_p\langle \phi_{B_n}\rangle(E) =\Gamma_p\langle \psi_{B_n}\rangle(E) \leq C_{\rm cap}\frac{\mu(B(x,r))}{\Psi(x,r)}.
\]
Finally, for $x\in F_<$, we must have $\psi_{B_n}(x)>0$ and $0\leq 1-\sum_{i=1}^{n-1} \phi_{B_i}< 1$.  But, then by \eqref{eq:phipsi}, we have
\[
\phi_{B_n}(x)=1-\sum_{i=1}^{n-1} \phi_{B_i}(x) = 1-\sum_{i=1}^{n-1} \psi_{B_i}(x).
\]
Let $J_n=\{j\in I : j < n \text{ and } B_j \cap B_n \neq \emptyset\} \subset I$. By local finiteness, we have $|J_n|\leq D$ and thus for $x\in F_<$
\[
\phi_{B_n}(x)=1-\sum_{j\in J_n} \psi_{B_j}(x),
\]
and by Proposition \ref{prop:properties} and Lemma \ref{lem:localenergy} and the triangle inequality again
\[
\Gamma_p\langle \phi_{B_n}\rangle(F_<)^\frac{1}{p}=\Gamma_p\langle 1-\sum_{j\in J_n} \psi_{B_j}\rangle(F_<)^\frac{1}{p}\leq \sum_{j\in J_n} \Gamma_p\langle \psi_{B_n}\rangle(X)^\frac{1}{p}\leq D C_{\rm cap}^\frac{1}{p} \frac{\mu(B(x,r))^{\frac{1}{p}}}{\Psi(x,r)^{\frac{1}{p}}}.  
\]
Combining these estimates, the energy bound directly follows.
\end{proof}

\subsection{Discrete convolutions}

 %    The proof goes verbatim through: One shows that there are discrete convolutions with energy bounded uniformly by $\nu(X)$ converging in $L^p(X)$ to $h$. Thus by lower semi-continuity and closedness $h\in \cF_p$. This also is an argument of Haj{\l}asz-Koskela.

 The domain $\cF_p$ has an abstract description, but once a Poincar\'e inequality and capacity bounds are assumed, it is easier to prove membership in $\cF_p$. Here, we characterize membership in $\cF_p$ in terms of a Poincar\'e type inequality. The characterization and its proof are adapted from \cite{Koskela1998}, see also \cite[Theorem 10.3.4]{HKST}. We note that this is quite closely related to Korevaar-Schoen-type characterizations used in a similar fashion in \cite{muruganreflected}. In fact, Korevaar-Schoen spaces were discussed by Koskela in \cite{Koskela1998}, and could be used here too. Many other works study the relationship between Korevaar-Schoen energies and the domain $\cF_p$, see e.g. \cite{KSTkorevaar,muruganreflected, murugan2023first,KumSturm,shimizu,ASEBShimizu, GHLKorevaar}.

For $\epsilon>0$ is positive, a set $N\subset X$ is called $\epsilon$-separated if for each $x,y\in N$ with $x\neq y$ we have $d(x,y)\geq \epsilon$. An $\epsilon$-net is a $\epsilon$-separated set $N$ that is maximal with respect to set-inclusion.

\begin{lemma}\label{lem:PIcharSob} Assume $\cCap_p(\Psi)$. Let $h\in L^p(X)$ have compact support s.t. there is a constant $\Lambda$ and a finite measure $\nu$ and a $\delta>0$ s.t. for all balls $B\subset X$ with $\rad(B) \leq \delta$ we have
\[
     \fint_B |h - h_B|^p d\mu \leq \frac{\Psi(x,r)}{\mu(B(x,r))} \nu(\Lambda B).
    \]
Then $h\in \cF_p$ with $\Gamma_p\langle h\rangle \lesssim \nu$.
\end{lemma}
\begin{proof} By scaling we may assume $\delta\geq 4$. Let $N_k$ be a $2^{-k}$-net in $X$ for $k\in \N$. For each $n\in N_k$, and let $\cB_k = \{B(n,2^{-k}):n \in N_k\}.$ Let $\phi_B$ be the Sobolev partition of unity given by Lemma \ref{lem:Sobolevpartofunity}. Define
\[
h_k = \sum_{B\in \cB_k} h_B \phi_B .
\]
We will show that $h_k\in \cF_p$ with the desired energy bound. Since the support of $h$ is compact, this sum always has finitely many non-zero terms. Similarly, all the sums below will be finite in this same sense.

%First, Let $\cB_{n,k} \subset \cB_k$, for $n\in\N$ be a sequence of increasing finite sets s.t. $\bigcup_{n\in \N} \cB_{n,k}=\cB_k$, and let
%\[
%h_{n,k} = \sum_{B\in \cB_{n,k}} h_B \phi_B.
%\]
As a finite sum of functions in $\cF_p$, we have $h_{k}\in \cF_p$ for all $k \in \N$. Further, 
\begin{align*}
\int |h_{k}-h|^p d\mu  &=\int \left|\sum_{B\in \cB_k} (h_B-h(x)) \phi_B\right|^p d\mu  \\
&\lesssim \int \sum_{B\in \cB_k} |h_B-h(x)|^p |\phi_B| d\mu  & \text{$\cB_k$ is locally finite}\\
&\lesssim \sum_{B\in \cB_k} \Psi(B)\nu(2\Lambda B)  & \text{assumption}.
\end{align*}
Since $\Psi$ is a regular scale function, we have $\lim_{r\to 0}\Psi(x,r)\to 0$ locally uniformly on all bounded sets, and thus $\int |h_{k}-h|^p d\mu\to_{k\to \infty} 0$ and $h_k$ converges to $h$ in $L^p(X).$

Finally, we estimate the energies of $h_{k}$. Let $A\subset X$ be a closed set and let $A_k = \{y \in X: d(y,A)\leq 2^{2-k}\Lambda\}$. First, if $B\cap B' \neq \emptyset$ for $B,B'\in \cB_k$, then $B'\subset 4B$ and the doubling property and the Poincar\'e inequality imply that
\begin{equation}\label{eq:neighborbound}
|h_B-h_{B'}|\leq \frac{1}{|B|}\int_{4B} |h-h_B| d\mu \lesssim \frac{\Psi(x,r)}{\mu(B(x,r))} \nu(\Lambda B).
\end{equation}
By applying these, we give a uniform bound for the energy of $h_{k}$:
\begin{align*}
\Gamma_p\langle h_{k}\rangle(A)  &\leq \sum_{B\in \cB_{k}} \Gamma_p\langle h_{k}\rangle(B\cap A) & \hspace{-1cm}\text{balls cover the set}\\
&\leq \sum_{B\in \cB_{k}} \Gamma_p\langle \sum_{\substack{B' \cap B \neq \emptyset \\ B'\in \cB_{k} }} h_{B'} \psi_{B'}\rangle(B\cap A)  & \hspace{-1cm}\text{locality}\\
&\leq \sum_{B\in \cB_{k}} \Gamma_p\langle \sum_{\substack{B' \cap B \neq \emptyset \\ B'\in \cB_{k} }} (h_{B'}-h_B) \psi_{B'}\rangle(B\cap A)  & \hspace{-1cm}\text{partition of unity}\\
&\lesssim \sum_{\substack{B\in \cB_{k} \\ B\cap A \neq \emptyset}} \sum_{\substack{B' \cap B \neq \emptyset \\ B'\in \cB_{k} }}\Gamma_p\langle  (h_{B'}-h_B) \psi_{B'}\rangle(B)  & \hspace{-1cm}\text{locally finiteness}\\
&\lesssim \sum_{\substack{B\in \cB_{k} \\B\cap A \neq \emptyset}} \sum_{\substack{B' \cap B \neq \emptyset \\ B'\in \cB_{k} }}|h_{B'}-h_B|^p\frac{\mu(B')}{\Psi(B')}  & \\
&&\hspace{-13em} \text{ homogeneity and Sobolev partition of unity}\\
&\lesssim \sum_{\substack{B\in \cB_{n,k} \\ B\cap A \neq \emptyset}} \sum_{\substack{B' \cap B \neq \emptyset \\ B'\in \cB_{k} }}|h_{B'}-h_B|^p\frac{\mu(B(x,r))}{\Psi(x,r)}  & \hspace{-1cm}\text{doubling}\\
&\lesssim \sum_{\substack{B\in \cB_{k} \\ B\cap A \neq \emptyset}} \nu(\Lambda B)  & \hspace{-1cm}\text{\eqref{eq:neighborbound}},\\
&\lesssim \nu(A_k).  & 
\end{align*}
In the final step we used the fact that metric doubling implies that the balls $\Lambda\cB_{n,k}$ have finite overlap.

By using $A=X$, we get $\Gamma_p\langle h_{k}\rangle(X)\lesssim \nu(X)$. Sending $k\to \infty$ and using lower semi-continuity, we get $h\in \cF_p$ with $\Gamma_p\langle h\rangle(X)\lesssim \nu(X)$. For general closed sets $A$, we get $\Gamma_p\langle h\rangle(A)\lesssim \nu(A_k)$ for every $k\in \N$, and then sending $k\to \infty$ yields $\Gamma_p\langle h\rangle(A)\lesssim \nu(A)$. By Borel regularity, we get the claim.

\end{proof}

\subsection{Whitney cover}

If $\Omega$ is an open set s.t. $X\setminus \Omega\neq \emptyset$, then a $(C_D,\Lambda)$-Whitney cover for $\Omega$ is a countable collection $\cB=\{B(x_i,r_i) : I\}$ for which
\begin{enumerate}
    \item $\Lambda^3 r_i/2\leq d(x_i, X\setminus \Omega) \leq \Lambda^3 r_i$, 
    \item $\Lambda^2\cB$ is $C_D$-bounded,
    \item $\Omega = \bigcup_{B\in \cB} B$.
\end{enumerate}

\begin{lemma}\label{lem:whitneycover} Assume that $X$ is $C_M$-metric doubling. For every $\Lambda \geq 8$ there exists a $C_D=C_D(C_M,\Lambda )$, such that for every $\Omega\subsetneq X$ there exists a $(C_D,\Lambda )$-Whitney cover $\cB$. 
\end{lemma}
\begin{proof}
    Let $i\in \Z$ and $\Omega_i = \{x\in \Omega : 2^{i-1} \leq d(x,X\setminus \Omega) \leq 2^{i}\}$. Let $N_i\subset \Omega_i$ be a $2^{i-3}\Lambda^{-3}$-net for $\Omega_i$. Define $\cB = \{ B(n, 2^{i} \Lambda^{-3}) : n\in N_i, i \in \mathbb{Z}\}$.

    For each $B(x,r)=B(n, 2^i \Lambda^{-3})\in \cB$ we have $r = 2^i \Lambda^{-3} \leq 2d(n,X\setminus \Omega)\Lambda^{-3}\leq 2 r$. This gives the first property of being a Whitney cover.

    Second, by metric doubling, the balls $\{B(n, 2^{i} \Lambda^{-1}) : n\in N_i\}$ have at most $D_0$-overlap for each $i$, with $D_0$ depending on $C_M$ and $\Lambda $, and $B(n,2^{i} \Lambda^{-1})\subset \Omega_i \cup\Omega_{i-1} \cup \Omega_{i+1}$. Thus, $B(n,2^{i} \Lambda^{-3})\cap B(m, 2^{j} \Lambda^{-3}) \neq \emptyset$ for $n\in N_i, m\in N_j$ implies $|i-j|\leq 2$. These together imply $C_D$-boundedness with a $C_D$ dependent on $D_0$.
\end{proof}

In the course of the proof of Lemma \ref{lem:whitneyblending}, we will need a small technical property of Whitney covers. In order to streamline the proof, we state it here.
\begin{lemma}\label{lem:whitneyproperties} Assume $\Lambda \geq 8$. If $B,B'\in \cB$ and $d(B,B')\leq 2\rad(B')$, then $B \subset 16 B' \subset 2\Lambda B'$ and $B' \subset 16 B \subset 2\Lambda B$, and $\rad(B)\leq 3 \rad(B')$ and $\rad(B')\leq 3 \rad(B)$ 
\end{lemma}
\begin{proof}
    Estimate 
    \begin{align*}
        \rad(B)&\leq 2\Lambda^{-3}d(x_B,X\setminus \Omega) \\
        &\leq 2\Lambda^{-3}(d(x_{B'},X\setminus \Omega)+3\rad(B')+\rad(B)) \\
         &\leq 2\Lambda^{-3}(\Lambda^3\rad(B')+3\rad(B')+\rad(B)).
    \end{align*}
    Then $\rad(B)\leq (2+3/\Lambda^3)(1-2\Lambda^{3})^{-1} \leq 3\rad(B')$.    Thus, $B\subset 9B'$ and
    \begin{align*}
        \rad(B')&\leq 2\Lambda^{-3}d(x_{B'},X\setminus \Omega) \\
        &\leq 2\Lambda^{-3}(d(x_{B},X\setminus \Omega)+2\rad(B')+\rad(B)) \\
         &\leq 2\Lambda^{-3}(\Lambda^3\rad(B)+2\rad(B)) + 2\Lambda^{-3} \rad(B').
    \end{align*}
    Moving the term $\rad(B')$ to the left hand side we obtain $\rad(B')\leq (1-2\Lambda^{-3})^{-1}(2+4\Lambda^{-3})\rad(B) \leq 3\rad(B)$. Also, $B' \subset B(x_B, \rad(B)+4\rad(B'))\subset 16 B$.
\end{proof}

\subsection{Blending procedure}

We first state a slightly more genral version of Lemma \ref{lem:whitneyblending}.

\begin{lemma}\label{lem:whitneyblending-stronger} Assume that $\cCap_p(\Psi)$ and $\PI_p(\Psi)$ hold. For every $\eta\in (0,1)$ there exists a $C_{\rm WB}=C_{\rm WB}(C_{\rm cap}, C_{D}, C_{\PI}, C_{\Psi}, \beta_\pm)>0$ s.t. the following holds. Let $f,g\in \cF_p$ and $B_0$ is a ball in $X$, then, there exists a $h\in \cF_p$ such that $h|_{B_0}=f|_{B_0}$, $h|_{X\setminus (1+\eta)B_0} = g|_{X\setminus (1+\eta)B_0}$ with
\[
\Gamma_p\langle h\rangle(2B_0)\leq C_{\rm WB}\frac{\mu(B_0)}{\Psi(B_0)} \fint_{(1+\eta)B_0}|f-g|^pd\mu + C_{\rm WB} \Gamma_p\langle f\rangle(2B_0)+C_{\rm WB}\Gamma_p\langle g\rangle(2B_0)
\]
\end{lemma}
The proof of Lemma \ref{lem:whitneyblending} follows by setting $\Psi(x,r)=r^\beta$, and thus it suffices to prove this stronger form. Our argument will be based on using averages of $f$ and $g$ combined through a partition of unity. This slightly complicated function $h$ is then shown to lie in $L^p(X)$ by a direct calculation. Finally, one shows $h\in \cF_p$ with the desired energy bound. This requires us to apply Lemma \ref{lem:PIcharSob}, and to find an appropriate measure $\nu$ to bound the oscillations $\fint_B |h-h_B|^p d\mu$. This step is the most technical, and involves a number of cases depending on the size and positioning of $B$. We note that a very similar analysis was previously implemented in \cite{muruganreflected}, where a Korevaar-Schoen-type characterization of $f\in \cF_p$ was used. It seems likely that the proof here could be also completed by a Korevaar-Schoen-type argument to a show that $h\in \cF_p$, but we choose to follow the related argument here that applies Lemma \ref{lem:PIcharSob}. 

\begin{proof}[Proof of Lemma \ref{lem:whitneyblending}]
\noindent{\textbf{Preparations:}} We start the proof with a little setup. The proof will be slightly simplified by setting $g=0$. We can assume this without loss of generality once we replace $f$ with $\tilde{f}=f-g$ and $g$ with $\tilde{g}=0$. The special case of the statement yields a function $\tilde{h}$ with the desired energy bound and the claim follows from the triangle inequality by choosing $h=\tilde{h}+g$. Thus, from now on, assume $g=0$. The claim is vacuous if $X\setminus (1+\eta)B_0=\emptyset$ (since then we can choose $h=f$). We may thus assume $X\setminus (1+\eta)B_0\neq \emptyset$. Let $\Omega = (1+\eta)B_0\setminus \overline{B_0}$, and let $\cB$ be the $(C_D,\Lambda )$-Whitney cover given in Lemma \ref{lem:whitneycover} for $\Omega$. 

\vskip1em

\noindent{\textbf{Definition of $h$:}} Let $\phi_B$ be the Sobolev partition of unity for $\cB$ given by Lemma \ref{lem:Sobolevpartofunity}. For each ball $B\in \cB$ define $c_{B}=f_B$, if $B=B(x,r)$ satisfies $d(x,B_0)\leq \eta/2 \rad(B_0)$. Let $\cB_0$ be the collection of these balls, and let $\cB_1=\cB \setminus \cB_0$ and let $c_{B}=0$ for $B\in \cB_1$. 
    
    Define:
    \[
    h(x)=\begin{cases}f(x) & x\in \overline{B_0} \\ 0 & x\not\in X\setminus (1+\eta)B_0 \\ \sum_{B\in \cB} \psi_{B} c_{B} & x\in \Omega\end{cases}.
    \]

\vskip1em

\noindent{\textbf{Proof of the $L^p(X)$ bound:}} Compute:
    \begin{equation}\label{eq:lpbound}
    \int_X |h|^p d\mu \leq \int_{B_0} |f|^p d\mu + \int_\Omega| \sum_{B\in \cB}\psi_{B} c_{B}|^p d\mu.
    \end{equation}
    Next,
    \begin{align*}
        \int_\Omega| \sum_{B\in \cB}\psi_{B} c_{B}|^p d\mu &\lesssim \sum_{B\in \cB} |c_{B}|^p \mu(2B) & \text{ bounded overlap of $2\cB$}\\
        &\leq \sum_{B\in \cB} |f_B|^p\mu(2B) \\
        &\lesssim \sum_{B\in \cB} |f_B|^p \mu(B)& \text{ doubling} \\
        &\lesssim \int_\Omega |f|^p d\mu. & \text{ bounded overlap of $2\cB$}
    \end{align*}
By combining this with \eqref{eq:lpbound}, we get $h\in L^p(X)$ and
\begin{equation}\label{eq:lpboundfinal}
  \int_X |h|^p d\mu \lesssim \int_{(1+\eta)B_0} |f|^p d\mu  
\end{equation} 

\noindent{\textbf{Proof that $h\in \cF_p$ and a bound for $\Gamma_p\langle h\rangle$:}} This step will be based on Lemma \ref{lem:PIcharSob}, and uses an auxiliary measure $\nu$ defined as follows
\begin{align*}
    \nu:= &\sum_{B\in \cB} \frac{\mu(B)}{\Psi(B)} \Gamma_p\langle f\rangle(2\Lambda^2B) \Gamma_p\langle\psi_{B}\rangle + \Gamma_p\langle f\rangle \\
    &\quad \quad +    \fint_{(1+\eta)B_0}|f|^p d\mu \left(\sum_{\substack{B\in \cB \\ \rad(B)\geq \eta/(12\Lambda^3) }} \Gamma_p\langle \psi_B\rangle\right).
\end{align*}

    The balls $2\Lambda^2\cB$ have bounded overlap by Lemma \ref{lem:whitneycover}. This with doubling implies that there are at most finitely many balls $B\in \cB$ with $\rad(B)\geq \eta/(12\Lambda^3)$, and their number is bounded by the doubling constant and $\Lambda $. Thus, the fact that $\psi_B$ is Sobolev partition of unity implies that $\nu$ is a finite measure. 
    
    Choose $C_\delta>0$ so small that  $(12+6\Lambda^3)C_\delta\leq \eta/2$ holds. The claim then follows from Lemma \ref{lem:PIcharSob} and the properties of a Sobolev partition of unity, once one establishes that there exist constants $C_\nu>0$ and $\mathbf{\Lambda }\geq 1$ s.t all $B\subset X$ with   $\rad(B)\leq C_\delta \rad(B_0)$  satisfy
    \[
     \fint_B |h - h_B|^p d\mu \leq C_{\nu}\frac{\Psi(B)}{\mu(B)} \nu(\mathbf{\Lambda}B).
    \]
    It will be convenient to prove instead that there exists a constant $m_B$ s.t.
\begin{equation}\label{eq:hpinu}
     \fint_B |h - m_B|^p d\mu \leq C_{\nu}\frac{\Psi(B)}{\mu(B)} \nu(\mathbf{\Lambda}B).
    \end{equation}
    This implies the main claim, since
    \[
     \fint_B |h - h_B|^p d\mu \leq 2^p \fint_B |h-m_B|^p + |m_B-h_B|^p d\mu \leq 2^{p+1} \fint_B |h-m_B|^p d\mu.
    \]

    There are two cases to consider in proving \eqref{eq:hpinu}.  The main cases are the following.
    \begin{enumerate}
        \item[\textbf{A)}] $B\subset B_0$ or $B\subset X\setminus (1+\eta)B_0$ 
        \item[\textbf{B)}] $B\cap \Omega \neq \emptyset$
    \end{enumerate}
    The final case has a further subdivisions, that we will execute later.

    \vskip1em
    
    \noindent \textbf{Case \textbf{A)}}:
     If $B\subset B_0$, then $h|_B=f|_B$ and $m_B=f_B$, and the claim follows from the Poincare\'e inequality for $f$. If on the other hand $B\subset X\setminus (1+\eta)B_0$:, then $h|_B=0$ and $m_B=0$, and the claim follows since the left hand side in \eqref{eq:hpinu} vanishes.

     \vskip1em
    
%    \noindent \textbf{Case \textbf{C)}:}
%     First, consider the case $\rad(B)>C_\delta \rad(B_0)$. In this case, set $m_B=0$, and note
%    \begin{align*}
%    \fint_B |h|^p d\mu &= \frac{1}{\mu(B)} \int_{(1+\eta)B_0} |h|^p d\mu & \text{$h=0$ outside } (1+\eta)B_0 \\
%    &\lesssim \frac{1}{\mu(B_0)} \int_{(1+\eta)B_0} |f|^p d\mu & \text{doubling and \eqref{eq:lpboundfinal}} \\
%    &\lesssim \frac{\Psi(B_0)}{\mu(B_0)} \nu(2B_0) & \text{ definition of $\nu$} \\
%    &\lesssim \frac{\Psi(B)}{\mu(B_0)} \nu(\mathbf{\Lambda}B),
%    \end{align*}
%    once $\mathbf{\Lambda} \geq 3(1+\eta)C_\delta^{-1}$. Thus, \eqref{eq:hpinu} holds.

\vskip1em
    
    \noindent \textbf{Case \textbf{B)}:}
    In this case  $B\cap \Omega \neq \emptyset$ and $\rad(B)\leq C_\delta \rad(B_0)$ (which holds by earlier restrictions). If $B\cap (\bigcup_{B' \in \cB_0} B')=\emptyset$, then \eqref{eq:hpinu} holds by setting $m_B=0$ since $h|_B=0$. Thus, we are left with the cases where there exists a $B_*\in \cB_0$ s.t. $B_* \cap B \neq \emptyset$. Fix such a ball $B_*$ for the following argument. There may be many such balls, but any choice suffices. It will affect the division to cases in the next paragraph.

    Depending on the size of the chosen $B_*$ we subdivide into two subcases: \textbf{B.1)} $\rad(B_*)\leq \rad(B)$ and \textbf{B.2)} $\rad(B_*)>\rad(B)$. 
    %{\color{blue}MM: The subcases are not clear. Does case (C.1) mean `$\rad(B_*)\leq \rad(B)$ for all  $B_*\in \cB_0$ such that $B_* \cap B \neq \emptyset$' or does it mean `there exists $B_*\in \cB_0$ such that $\rad(B_*)\leq \rad(B)$ and  $B_* \cap B \neq \emptyset$'? I think you want the former   for the argument below to apply. Similar comment also applies to case (C.2).}

\vskip1em
    
    \noindent \textbf{ Subcase \textbf{B.1)} $\rad(B_*)\leq \rad(B)$:} In this case
    \begin{align}
    d(x_B, B_0)+\rad(B)&\leq d(x_{B_*}, B_0) + 3\rad(B) \nonumber \\
    &\leq \Lambda^3 \rad(B_*) + 3\rad(B) \leq (3+\Lambda^3)C_\delta \rad(B_0). \label{eq:distxbbound}
    \end{align}
    Here, we used $B_*\in \cB_0$ to conclude $d(x_{B_*}, B_0)=d(x_{B_*}, X\setminus \Omega)$ and the definition of the Whitney cover. 

    Similarly
    \begin{align}
    d(x_B, B_0)&\leq d(x_{B_*}, B_0) + 3\rad(B) \nonumber \\
    &\leq \Lambda^3 \rad(B_*) + 3\rad(B) \leq (3+\Lambda^3)\rad(B). \label{eq:distxbbound2}
    \end{align}
    Recall that $C_\delta$ is chosen so small that $(12+6\Lambda^3)C_\delta \leq \eta/2$. From \eqref{eq:distxbbound} we get $B \subset (1+\eta)B_0$.  If $B'\in \cB$ is such that $B' \cap B \neq \emptyset$, then 
    \[
    \rad(B') \leq 2d(x_{B'}, B_0)/\Lambda  \leq 2\frac{\rad(B')+\rad(B)+d(x_B, B_0)}{\Lambda }.
    \]
    Here we used the definition of a Whitney cover and that $d(x_{B'},X\setminus \Omega)\leq d(x_{B'}, B_0)$, which follows from $B_0\subset X\setminus \Omega$.
    This and \eqref{eq:distxbbound} gives
    \begin{equation}\label{eq:bpradboundfirst}
    \rad(B') \leq 2\frac{d(x_B, B_0)+\rad(B)}{\Lambda -2}\leq 2\frac{\Lambda^3 \rad(B)+\rad(B)}{\Lambda -2}\frac{6+3\Lambda^3}{\Lambda -2}C_\delta \rad(B_0), 
    \end{equation}
    and combing this with \eqref{eq:distxbbound2} gives
    \[
    d(x_{B'}, B_0)\leq d(x_B, B_0)+\rad(B)+\rad(B') \leq (12+6\Lambda^3)C_\delta \rad(B_0) \leq \eta/2\rad(B_0).
    \]
    By definition, $B'\in \cB_0$. Therefore, all the balls $B'\in \cB$ that intersect $B$ satisfy $B'\in \cB_0$. Moreover, \eqref{eq:bpradboundfirst} and \eqref{eq:distxbbound2} yield
    \begin{equation}\label{eq:radbpbound}
    \rad(B')\leq \frac{(\Lambda^3+4)\rad(B)}{\Lambda -2}.
    \end{equation}
    Thus, $B'\subset LB$ with $L=3\Lambda^3+10$. 
    
    Set $m_B= f_{B}$. Since $B \subset (1+\eta)B_0$ we have
    \begin{align}
        \fint_B |h-h_B| d\mu &\leq  \frac{1}{\mu(B)} \int_{B_0\cap B} |f-f_B| d\mu + \frac{1}{\mu(B)} \int_{\Omega\cap B} |h-f_B| d\mu. \label{eq:hbound}
        \end{align}
        For the first term, the Poincar\'e inequality gives
        \[
        \frac{1}{\mu(B)} \int_{B_0\cap B} |f-f_B|^p d\mu \lesssim \frac{\Psi(B)}{\mu(B)}\Gamma_p\langle f\rangle(\Lambda B) \lesssim \frac{\Psi(B)}{\mu(B)} \nu(\Lambda B),
        \]
        and thus it suffices to bound the second term in  \eqref{eq:hbound}.
        \begin{align*}
        \frac{1}{\mu(B)} &\int_{\Omega\cap B} |h-f_B|^p d\mu \lesssim  \frac{1}{\mu(B)} \int_{\Omega\cap B} \sum_{B' \in \cB_0}|f_{B'}-f_B|^p \psi_{B'} d\mu & \hspace{-5em}\text{bounded overlap} \\
        &\lesssim  \frac{1}{\mu(B)} \sum_{\substack{B' \in \cB_0 \\ B'\cap B\neq \emptyset}}|f_{B'}-f_B|^p \mu(2B') & \hspace{-2em}\text{ support of $\psi_{B'}$ in $2B'$} \\
        &\lesssim  \frac{1}{\mu(B)} \sum_{\substack{B' \in \cB_0 \\ B'\cap B\neq \emptyset}}\int_{B'}|f(x)-f_B|^pd\mu & \text{Jensen} \\
        &\lesssim  \frac{1}{\mu(B)} \int_{LB}|f(x)-f_B|^pd\mu & \hspace{-20em}\text{$B'\subset LB$ and bounded overlap} \\
        &\lesssim  \frac{1}{\mu(LB)} \int_{LB}|f(x)-f_{LB}|^pd\mu & \text{Lemma \ref{lem:doublingballs} and doubling}\\
        &\lesssim   \frac{\Psi(LB)}{\mu(LB)} \Gamma_p\langle f\rangle(\Lambda LB) & \text{Poincar\'e}\\
        \end{align*}
        Inequality \eqref{eq:hpinu} now follows by doubling, the regularity of $\Psi$ and by taking $\mathbf{\Lambda}$ large enough.
\vskip1em

        \noindent \textbf{Subcase \textbf{B.2)} $\rad(B_*)>\rad(B)$:} This is the final case. Lemma \ref{lem:whitneyproperties} implies  that for all $B'\in \cB$ s.t. $B'\cap B \neq \emptyset$
        \begin{equation}\label{eq:radii}
        1/3 \rad(B_*)\leq \rad(B')\leq 3\rad(B_*),
        \end{equation}
        \begin{equation}\label{eq:inclusion}
        B' \subset 2\Lambda  B_* \quad \quad \text{ and } \quad \quad B_* \subset 2\Lambda  B'.
        \end{equation}
        Also $B\subset 3B_*$ and thus $B\subset \Omega$. Thus 
        \[
        h|_B=F|_B,
        \]
        with $F_B=\sum_{\substack{B'\in \cB_0 \\ B'\cap B\neq \emptyset}} c_{B'} \psi_{B'}$.
        The function $F$ is a finite sum of functions in $\cF_p$, where the number of terms is bounded by $C_D$ by local boundedness and \eqref{eq:inclusion}. Thus
        \begin{align}
        \fint_B |h-h_{B}| d\mu &\lesssim \frac{\Psi(B)}{\mu(B)} \Gamma_p\langle F\rangle(\Lambda B).\label{eq:finalbound}
        \end{align}
        This yields \eqref{eq:hpinu} once we bound the energy of $F$ by $\nu$. On bounding the energy of $F$ we will still need to consider two subcases.
        
        \textbf{First, assume that all balls $B'\in \cB$ s.t. $B'\cap B\neq \emptyset$ satisfy $B'\in \cB_0$.}  By local boundeness and the triangle inequality Remark \ref{rmk:local} and Lemma \ref{lem:localenergy}, we get 
        \begin{align}
        \Gamma_p\langle F\rangle = \Gamma_p\langle F-f_{B_*}\rangle \lesssim \sum_{\substack{B'\in \cB_0 \\ B'\cap B\neq \emptyset}}|c_{B'}-f_{B_*}|^p \Gamma_{p}\langle \psi_{B'}\rangle. \label{eq:energybound}
        \end{align}

        Lemma \ref{lem:doublingballs} together with estimates \eqref{eq:inclusion} and \eqref{eq:radii} yield
        \[
        |c_{B'}-f_{B_*}|^p=|f_{B'}-f_{B_*}|^p \lesssim \frac{\Psi(B')}{\mu(B')} \Gamma_p\langle f\rangle(2\Lambda^2 B'). 
        \]
        Combined with \eqref{eq:energybound}, we get
    \begin{align*}
        \Gamma_p\langle F\rangle \lesssim \sum_{B'\in \cB} \frac{\Psi(B')}{\mu(B')} \Gamma_p\langle f\rangle(2\Lambda^2 B') \Gamma_{p}\langle \psi_{B'}\rangle \leq \nu,
    \end{align*}
    and the claim follows from \eqref{eq:finalbound}.

    If the previous assumption doesn't hold, then we can \textbf{assume that there is some $B_{**}\in \cB_1$ s.t. $B_{**}\cap B\neq \emptyset$}. We have
    \[
    d(x_{B_{**}}, B_0)> \eta/2 \rad(B_0),
    \]
    and
    \[
    d(x_{B_*}, B_0)\leq \eta/2 \rad(B_0).
    \]
    Then, \eqref{eq:radii} implies that
    \[
    d(x_{B_{**}}, x_{B_*}) \leq \rad(B_{**}) + \rad(B_{*}) + 2\rad(B) \leq 6\rad(B_*).
    \]
    Thus
    \[
    d(x_{B_*}, B_0) \geq d(x_{B_{**}}, B_0)-6\rad(B_*) \geq(\eta/2-6C_\delta) \rad(B_0)\geq \eta/4 \rad(B_0),
    \]
    once $6C_\delta\leq \eta/4$. Therefore, 
    \[
    \rad(B_*)\geq \eta/(\Lambda^3 4) \rad(B_0).
    \]
    By \eqref{eq:radii} we also have
    \begin{equation}\label{eq:radlower}
    \rad(B')\geq \eta/(12\Lambda^3) \rad(B_0),
    \end{equation}
    for all $B'\in \cB$ with $B' \cap B\neq \emptyset$. Then 
    
    \begin{align*}
        \Gamma_p\langle F\rangle &\lesssim \sum_{\substack{B'\in \cB_0 \\ B'\cap B\neq \emptyset}}|f_{B'}|^p \Gamma_{p}\langle \psi_{B'}\rangle  \\
        &\lesssim \sum_{\substack{B'\in \cB_0 \\ B'\cap B\neq \emptyset}}\fint_{B'}|f|^p d\mu \Gamma_{p}\langle \psi_{B'}\rangle  \\
        &\lesssim \sum_{\substack{B'\in \cB_0 \\ \rad(B')\geq  \eta/(\Lambda^3 12) \rad(B_0)}}\fint_{(1+\eta)B_0}|f|^p d\mu \Gamma_{p}\langle \psi_{B'}\rangle  & \eqref{eq:radlower} \text{ and doubling }   \\
        &\lesssim \nu & \text{ definition of $\nu$ }.
        \end{align*}
    Again, this together with the bound in \eqref{eq:finalbound} implies the claim.
\end{proof}

\bibliographystyle{acm}
\bibliography{clp}

\end{document}